\documentclass[english,british]{amsart}
\usepackage[T1]{fontenc}
\usepackage[latin9]{inputenc}
\usepackage[a4paper]{geometry}
\usepackage{amsthm}
\usepackage{amssymb}
\usepackage{setspace}
\makeatletter
\numberwithin{equation}{section}
\numberwithin{figure}{section}
\theoremstyle{plain}
\newtheorem{thm}{\protect\theoremname}[section]
  \theoremstyle{remark}
  \newtheorem{rem}[thm]{\protect\remarkname}
  \theoremstyle{remark}
  \newtheorem*{rem*}{\protect\remarkname}
  \theoremstyle{plain}
  \newtheorem{cor}[thm]{\protect\corollaryname}
  \theoremstyle{plain}
  \newtheorem{fact}[thm]{\protect\factname}
  \theoremstyle{definition}
  \newtheorem{defn}[thm]{\protect\definitionname}
  \theoremstyle{remark}
  \newtheorem*{notation*}{\protect\notationname}
  \theoremstyle{plain}
  \newtheorem{lem}[thm]{\protect\lemmaname}
  \theoremstyle{remark}
  \newtheorem{notation}[thm]{\protect\notationname}
  \theoremstyle{plain}
  \newtheorem{prop}[thm]{\protect\propositionname}

\usepackage{amscd}
\usepackage{mathptmx}
\usepackage{bbm}

\newcommand{\e}{\mathrm e}
\newcommand{\1}{1}
\newcommand{\N}{\mathbb{N}}
\newcommand{\F}{F}
\newcommand{\Z}{\mathbb{Z}}

\newcommand{\R}{\mathbb{R}}

\renewcommand{\Pi}{\pi}
\renewcommand{\emptyset}{\varnothing}
\renewcommand{\hat}{\widehat}

\DeclareMathOperator*{\Int}{Int}
\newcommand\id{\mathrm{id}}

\DeclareMathOperator*{\card}{card}

\DeclareMathOperator*{\Con}{Con}

\newcommand{\NwithoutId}{N \setminus \{ \id \}}

\usepackage{hyperref}
\@ifundefined{textmu}
 {\usepackage{textcomp}}{}

\makeatother

\usepackage{babel}
  \addto\captionsbritish{\renewcommand{\corollaryname}{Corollary}}
  \addto\captionsbritish{\renewcommand{\definitionname}{Definition}}
  \addto\captionsbritish{\renewcommand{\factname}{Fact}}
  \addto\captionsbritish{\renewcommand{\lemmaname}{Lemma}}
  \addto\captionsbritish{\renewcommand{\notationname}{Notation}}
  \addto\captionsbritish{\renewcommand{\propositionname}{Proposition}}
  \addto\captionsbritish{\renewcommand{\remarkname}{Remark}}
  \addto\captionsbritish{\renewcommand{\theoremname}{Theorem}}
  \addto\captionsenglish{\renewcommand{\corollaryname}{Corollary}}
  \addto\captionsenglish{\renewcommand{\definitionname}{Definition}}
  \addto\captionsenglish{\renewcommand{\factname}{Fact}}
  \addto\captionsenglish{\renewcommand{\lemmaname}{Lemma}}
  \addto\captionsenglish{\renewcommand{\notationname}{Notation}}
  \addto\captionsenglish{\renewcommand{\propositionname}{Proposition}}
  \addto\captionsenglish{\renewcommand{\remarkname}{Remark}}
  \addto\captionsenglish{\renewcommand{\theoremname}{Theorem}}
  \providecommand{\corollaryname}{Corollary}
  \providecommand{\definitionname}{Definition}
  \providecommand{\factname}{Fact}
  \providecommand{\lemmaname}{Lemma}
  \providecommand{\notationname}{Notation}
  \providecommand{\propositionname}{Proposition}
  \providecommand{\remarkname}{Remark}
\providecommand{\theoremname}{Theorem}

\begin{document}

\title{Conformal Fractals for  Normal Subgroups of Free Groups}

\author{Johannes Jaerisch}

\thanks{The author was supported by the research fellowship JA 2145/1-1 of
the German Research Foundation (DFG)}

\address{Department of Mathematics, Graduate School of Science, Osaka University,
1-1 Machikaneyama, Toyonaka, Osaka, 560-0043 Japan }

\email{jaerisch@cr.math.sci.osaka-u.ac.jp}


\begin{abstract}
We investigate subsets of a multifractal decomposition of the limit
set of a conformal graph directed Markov system, which is constructed
from the Cayley graph of a free group with at least two generators.
The subsets we consider are parametrised by a normal subgroup $N$
of the free group and mimic the radial limit set of a Kleinian group.
Our main results show that, regarding the Hausdorff dimension of these
sets, various results for Kleinian groups can be generalised. Namely,
under certain natural symmetry assumptions on the  multifractal decomposition,
we prove that, for a subset parametrised by $N$, the Hausdorff dimension
is maximal if and only if $\F_{d}/N$ is amenable and that the dimension
is greater than half of the maximal value. We also give a criterion
for amenability via  the divergence of the Poincar\'{e} series of
$N$. Our results are applied to the Lyapunov spectrum for normal
subgroups of Kleinian groups of Schottky type. 
\end{abstract}
\maketitle

\section{Introduction and Statement of Results}

In this paper we investigate the Hausdorff dimension of a class of
conformal fractals associated to normal subgroups of free groups.
Our main results show that these fractals share many interesting properties
with the radial (or conical) limit set of a Kleinian group. To define
the sets we consider, let $\F_{d}$ denote the free group generated
by $I:=\left\{ g_{1},g_{1}^{-1},\dots,g_{d},g_{d}^{-1}\right\} $
with $d\ge2$. The set of infinite reduced paths starting from the
identity in the Cayley graph of $\F_{d}$ with respect to $I$ is
given by $\Sigma:=\left\{ \tau=\left(\tau_{i}\right)\in I^{\N}:\tau_{i}\neq\tau_{i+1}^{-1}\right\} $.
Let $\Phi$ be a \emph{conformal graph directed Markov system $\Phi$
associated to $\F_{d}=\left\langle g_{1},\dots,g_{d}\right\rangle $}
with coding map $\pi_{\Phi}:\Sigma\rightarrow\R^{D}$, $D\ge1$ (see
Definition \ref{def:gdms-associated-to-freegroup-and-radiallimitsets}).
Such a system $\Phi$ consists of a set of contracting conformal maps
on $\R^{D}$, and each \emph{limit point} $\pi_{\Phi}\left(\tau\right)$
of $\Phi$ is obtained by successively applying maps from $\Phi$
according to the infinite path $\tau\in\Sigma$. We will define subsets
of the limit set $\pi_{\Phi}\left(\Sigma\right)$ as follows. For
a normal subgroup $N$ of $\F_{d}$, the \emph{symbolic radial limit
set }$\Lambda_{\mathrm{r}}\left(N\right)$ of $N$ consists of those
paths in $\Sigma$, for which the projection to the quotient graph
$\F_{d}/N$ visits some vertex infinitely often, that is, 
\[
\Lambda_{\mathrm{r}}\left(N\right):=\left\{ \tau\in\Sigma:\exists h\in\F_{d},\mbox{ such that }\tau_{1}\cdot\dots\cdot\tau_{n}\in hN\mbox{ for infinitely many }n\in\N\right\} .
\]
We will investigate the sets $\pi_{\Phi}\left(\Lambda_{\mathrm{r}}\left(N\right)\right)$
and $\pi_{\Phi}\left(\Lambda_{\mathrm{r}}\left(N\right)\cap\mathcal{F}\left(\alpha,\Phi,\psi\right)\right)$
in Theorem \ref{thm:maintheorem-limitset} and \ref{thm:maintheorem},
where the multifractal level sets \foreignlanguage{english}{$\mathcal{F}\left(\alpha,\Phi,\psi\right)$
are for} $\alpha\in\R$ given by 
\[
\mathcal{F}\left(\alpha,\Phi,\psi\right):=\left\{ \tau\in\Sigma:\lim_{n\rightarrow\infty}\frac{\sum_{i=0}^{n-1}\psi\left(\sigma^{i}\left(\tau\right)\right)}{\sum_{i=0}^{n-1}\zeta\left(\sigma^{i}\left(\tau\right)\right)}=\alpha\right\} ,
\]
the potential $\psi:\Sigma\rightarrow\R$ is H\"older continuous,
$\zeta:\Sigma\rightarrow\R$ is the geometric potential of $\Phi$
and $\sigma:\Sigma\rightarrow\Sigma$ refers to the left shift map.
For the definition of the symbolic uniformly radial limit set $\Lambda_{\mathrm{ur}}\left(N\right)$
we refer to Definition \ref{def:gdms-associated-to-freegroup-and-radiallimitsets}.
We refer to Section \ref{sec:Thermodynamic-Formalism-for} and \ref{sec:Graph-Directed-Markov}
for an introduction to symbolic thermodynamic formalism and graph
directed Markov systems.

To state our first main result, we have to make further definitions.
For each $n\in\N$, the set of admissible words of length $n$ is
given by $\Sigma^{n}:=\left\{ \omega\in I^{n}:\omega_{i}\neq\omega_{i+1}^{-1},1\le i\le n-1\right\} $.
For $\omega\in\Sigma^{n}$ and a function $f:\Sigma\rightarrow\R$
we define \foreignlanguage{english}{
\[
S_{\omega}f:=\sup_{\tau\in\left[\omega\right]}\sum_{i=0}^{n-1}f\left(\sigma^{i}\left(\tau\right)\right),\quad\mbox{where }\left[\omega\right]:=\left\{ \tau\in\Sigma:\tau_{1}=\omega_{1},\dots,\tau_{n}=\omega_{n}\right\} .
\]
For convenience, we denote by $\emptyset$ the unique word of length
zero and we set $S_{\emptyset}f:=0$. We let $\Sigma^{*}:=\bigcup_{n\in\N}\Sigma^{n}$
and we identify $\Sigma^{*}\cup\left\{ \emptyset\right\} $ with $\F_{d}$.
Let $N$ denote a normal subgroup of $\F_{d}$. The} \emph{Poincar\'{e}
series} \emph{of $\left(N,\Phi\right)$} and the \emph{exponent of
convergence} \emph{of $\left(N,\Phi\right)$} are for $u\in\R$ given
by 
\[
P_{N}\left(u,\Phi\right):=\sum_{\omega\in N}\e^{uS_{\omega}\zeta}\quad\mbox{and}\quad\delta_{N}:=\inf\left\{ u\in\R:P_{N}\left(u,\Phi\right)<\infty\right\} .
\]
More generally, for a H\"older continuous potential $\psi:\Sigma\rightarrow\R$,
we introduce the \emph{free energy function }of $\left(N,\Phi,\psi\right)$
which is for $\beta\in\R$ given by 
\[
t_{N}:\R\rightarrow\R,\quad t_{N}\left(\beta\right):=\inf\left\{ u\in\R:\sum_{\omega\in N}\e^{\beta S_{\omega}\psi+uS_{\omega}\zeta}<\infty\right\} .
\]
Let us also set $\delta:=\delta_{\F_{d}}$ and $t:=t_{\F_{d}}$. For
$\beta\in\R$, we say that $\left(N,\Phi,\psi\right)$ is of \emph{divergence
type in $\beta$} if $\sum_{\omega\in N}\e^{\beta S_{\omega}\psi+t_{N}\left(\beta\right)S_{\omega}\zeta}=\infty$.
We say that $\left(N,\Phi\right)$ is of \emph{divergence type} if
$\left(N,\Phi,0\right)$ is of divergence type in $0$, that is $P_{N}\left(\delta_{N},\Phi\right)=\infty$. 

We need the following notions of symmetry (cf. Definition \ref{def-asympt-symmetric}).
\foreignlanguage{english}{For each $n\in\N$ and $\omega\in\Sigma^{n}$}
we set \foreignlanguage{english}{$\left|\omega\right|:=n$ and} \foreignlanguage{english}{$\omega^{-1}:=\left(\omega_{n}^{-1},\dots,\omega_{1}^{-1}\right)$.
}We say that $\left(N,\Phi,\psi\right)$ is \emph{asymptotically symmetric},
if for all $\beta,u\in\R$ there exist $n_{0}\in\N$ and sequences
$\left(c_{n}\right)\in\left(\R^{+}\right)^{\N}$ and $\left(N_{n}\right)\in\N^{\N}$
with $\lim_{n}\left(c_{n}\right)^{1/n}=1$ and $\lim_{n}n^{-1}N_{n}=0$,
such that for each $g\in\F_{d}$ and for all $n\ge n_{0}$, 
\[
\sum_{\omega\in Ng:\left|\omega\right|=n}\e^{\beta S_{\omega}\psi+uS_{\omega}\zeta}\le c_{n}\sum_{\omega\in Ng^{-1}:n-N_{n}\le\left|\omega\right|\le n+N_{n}}\e^{\beta S_{\omega}\psi+uS_{\omega}\zeta}.
\]
If $\left(c_{n}\right)$ can be chosen to be bounded, for all $\beta,u\in\R$,
then $\left(N,\Phi,\psi\right)$ is called \emph{symmetric.} 

For $\beta\in\R$, we say that $\left(N,\Phi,\psi\right)$ is \emph{symmetric
on average in} $\beta$ if 
\[
\sup_{g\in\F_{d}}\limsup_{n\rightarrow\infty}\frac{\sum_{k=1}^{n}\sum_{\omega\in Ng:\left|\omega\right|=kp}\e^{\beta S_{\omega}\psi+t_{N}\left(\beta\right)S_{\omega}\zeta}}{\sum_{k=1}^{n}\sum_{\omega\in Ng^{-1}:\left|\omega\right|=kp}\e^{\beta S_{\omega}\psi+t_{N}\left(\beta\right)S_{\omega}\zeta}}<\infty,
\]
where $p:=\gcd\left\{ n\in\N:\exists\omega\in\Sigma^{n}\cap N\mbox{ such that }\omega_{n}\omega_{1}\neq\id\right\} $. 

We say that $\left(N,\Phi\right)$ is \emph{(asymptotically) symmetric}
if $\left(N,\Phi,0\right)$ is (asymptotically) symmetric, and $\left(N,\Phi\right)$
is \emph{symmetric on average }if $\left(N,\Phi,0\right)$ is symmetric
on average in $0$. 
\begin{thm}
\label{thm:maintheorem-limitset}Let $\Phi$ denote a conformal graph
directed Markov system associated to $\F_{d}$ with $d\ge2$. Then
the following holds for each non-trivial normal subgroup \emph{$N$
of $\F_{d}$. }
\begin{enumerate}
\item \label{enu:maintheorem-limitset-0}We have 
\begin{align*}
\dim_{H}\left(\pi_{\Phi}\left(\Lambda_{\mathrm{ur}}\left(N\right)\right)\right) & =\dim_{H}\left(\pi_{\Phi}\left(\Lambda_{\mathrm{r}}\left(N\right)\right)\right)=\delta_{N}>0.
\end{align*}

\item \label{enu:maintheorem-limitset-nonamenable-implies-gap-1}If $\F_{d}/N$
is non-amenable, then 
\[
\dim_{H}\left(\pi_{\Phi}\left(\Lambda_{\mathrm{r}}\left(N\right)\right)\right)<\dim_{H}\left(\pi_{\Phi}\left(\Lambda_{\mathrm{r}}\left(F_{d}\right)\right)\right).
\]

\item Suppose that $\left(N,\Phi\right)$ is asymptotically symmetric. \label{enu:maintheorem-limitset-2}

\begin{enumerate}
\item \label{enu:maintheorem-limitset-amenablesymmetric-fulldimension-1}If
$\F_{d}/N$ is amenable, then 
\[
\dim_{H}\left(\pi_{\Phi}\left(\Lambda_{\mathrm{r}}\left(N\right)\right)\right)=\dim_{H}\left(\pi_{\Phi}\left(\Lambda_{\mathrm{r}}\left(F_{d}\right)\right)\right).
\]

\item \label{enu:maintheorem-limitset-symmetriclowerbound-1}We have 
\[
\dim_{H}\left(\pi_{\Phi}\left(\Lambda_{\mathrm{r}}\left(N\right)\right)\right)\ge\dim_{H}\left(\pi_{\Phi}\left(\Lambda_{\mathrm{r}}\left(F_{d}\right)\right)\right)\big/2
\]
and strict inequality holds if $\left(N,\Phi\right)$ is symmetric. 
\end{enumerate}
\item Suppose that $\left(N,\Phi\right)$ is of divergence type.\label{enu:maintheorem-limitset-3}

\begin{enumerate}
\item \label{enu:maintheorem-limitset-divtype-amenable}Then $\F_{d}/N$
is amenable.
\item \label{enu:maintheorem-limitset-divtype-full-symmaverage}We have
\foreignlanguage{english}{\textup{$\dim_{H}\left(\pi_{\Phi}\left(\Lambda_{\mathrm{r}}\left(N\right)\right)\right)=\dim_{H}\left(\pi_{\Phi}\left(\Lambda_{\mathrm{r}}\left(F_{d}\right)\right)\right)$}\textup{\emph{
if and only if $\left(N,\Phi\right)$ is symmetric on average.}}}\emph{ }
\end{enumerate}
\end{enumerate}
\end{thm}
Our next aim is to investigate the set $\pi_{\Phi}\left(\Lambda_{\mathrm{r}}\left(N\right)\cap\mathcal{F}\left(\alpha,\Phi,\psi\right)\right)$.
To this end, let us define 
\[
\alpha_{-}:=\inf\left\{ \alpha\in\R:\mathcal{F}\left(\alpha,\Phi,\psi\right)\neq\emptyset\right\} \quad\mbox{and}\quad\alpha_{+}:=\sup\left\{ \alpha\in\R:\mathcal{F}\left(\alpha,\Phi,\psi\right)\neq\emptyset\right\} .
\]
By H\"older's inequality we see that the free energy function $t_{N}:\R\rightarrow\R$
of $\left(N,\Phi,\psi\right)$ is convex. We denote by $\partial t_{N}(\beta)$
the subdifferential of $t_{N}$ at $\beta$ (\cite[Section 23]{rockafellar-convexanalysisMR0274683}),
and we set $\partial t_{N}(A):=\bigcup_{\beta\in A}\partial t_{N}(\beta)$,
for $A\subset\R$. Let $\Int(A)$ denote the interior of a set $A\subset\R$.
We will always assume that $\alpha_{-}<\alpha_{+}$, which is equivalent
to the assumption that, for each $\left(a_{1},a_{2}\right)\in\R^{2}\setminus\left\{ 0\right\} $,
the unique Gibbs measures associated to $a_{1}\zeta$ and $a_{2}\psi$
are distinct (\cite{MR511655,MR1435198}). Then it is well-known that
$t$ is a real-analytic and strictly convex function which satisfies
$-\partial t\left(\R\right)=\left(\alpha_{-},\alpha_{+}\right)$.
Moreover, we have that $\mathcal{F}\left(\alpha,\Phi,\psi\right)=\emptyset$
if and only if $\alpha\notin\left[\alpha_{-},\alpha_{+}\right]$ (\cite{MR1738952}).
Note that, if $\alpha_{-}$ and $\alpha_{+}$ coincide,  then $\mathcal{F}\left(\alpha_{-},\Phi,\psi\right)=\Sigma$
and the analysis given in Theorem \ref{thm:maintheorem-limitset}
applies. 

To state our next main result, we define the convex conjugate of $t_{N}$
(\cite[Section 12]{rockafellar-convexanalysisMR0274683})  given by
\[
t_{N}^{*}:\R\rightarrow\R\cup\left\{ \infty\right\} ,\quad t_{N}^{*}\left(\alpha\right):=\sup_{\beta\in\R}\left\{ \beta\alpha-t_{N}\left(\beta\right)\right\} ,\quad\alpha\in\R.
\]

\begin{thm}
\label{thm:maintheorem}Let $\Phi$ denote a conformal graph directed
Markov system associated to $\F_{d}$ with $d\ge2$. Let $\psi:\Sigma\rightarrow\R$
be  H\"older continuous. Suppose that $\alpha_{-}<\alpha_{+}$. Then
the following holds for each non-trivial normal subgroup \emph{$N$
of $\F_{d}$. }
\begin{enumerate}
\item \label{enu:maintheorem-0}For each $\alpha\in-\Int\left(\partial t_{N}(\R)\right)\subset\left(\alpha_{-},\alpha_{+}\right)$
we have 
\[
\dim_{H}\left(\pi_{\Phi}\left(\Lambda_{\mathrm{ur}}\left(N\right)\cap\mathcal{F}\left(\alpha,\Phi,\psi\right)\right)\right)=\dim_{H}\left(\pi_{\Phi}\left(\Lambda_{\mathrm{r}}\left(N\right)\cap\mathcal{F}\left(\alpha,\Phi,\psi\right)\right)\right)=-t_{N}^{*}\left(-\alpha\right)>0.
\]
If $\left(N,\Phi,\psi\right)$ is asymptotically symmetric then $-\Int\left(\partial t_{N}(\R)\right)=\left(\alpha_{-},\alpha_{+}\right)$.
 
\item \label{enu:maintheorem-nonamenable-implies-gap}If $\F_{d}/N$ is
non-amenable, then for each $\alpha\in\left(\alpha_{-},\alpha_{+}\right)$,
\[
\dim_{H}\left(\pi_{\Phi}\left(\Lambda_{\mathrm{r}}\left(N\right)\cap\mathcal{F}\left(\alpha,\Phi,\psi\right)\right)\right)<\dim_{H}\left(\pi_{\Phi}\left(\mathcal{F}\left(\alpha,\Phi,\psi\right)\right)\right).
\]

\item Suppose that $\left(N,\Phi,\psi\right)$ is asymptotically symmetric
and let $\alpha\in\left(\alpha_{-},\alpha_{+}\right)$. \label{enu:maintheorem-2}

\begin{enumerate}
\item \label{enu:maintheorem-amenablesymmetric-fulldimension}If $\F_{d}/N$
is amenable, then 
\[
\dim_{H}\left(\pi_{\Phi}\left(\Lambda_{\mathrm{r}}\left(N\right)\cap\mathcal{F}\left(\alpha,\Phi,\psi\right)\right)\right)=\dim_{H}\left(\pi_{\Phi}\left(\mathcal{F}\left(\alpha,\Phi,\psi\right)\right)\right).
\]

\item \label{enu:maintheorem-symmetriclowerbound}We have 
\[
\dim_{H}\left(\pi_{\Phi}\left(\Lambda_{\mathrm{r}}\left(N\right)\cap\mathcal{F}\left(\alpha,\Phi,\psi\right)\right)\right)\ge\dim_{H}\left(\pi_{\Phi}\left(\mathcal{F}\left(\alpha,\Phi,\psi\right)\right)\right)\big/2
\]
and strict inequality holds if $\left(N,\Phi,\psi\right)$ is symmetric. 
\end{enumerate}
\item Let $\beta\in\R$ and suppose that $\left(N,\Phi,\psi\right)$ is
of divergence type in $\beta$. \label{enu:maintheorem-3}

\begin{enumerate}
\item \label{enu:enu:maintheorem-divtype-amenable}Then $\F_{d}/N$ is amenable.
\item \label{enu:maintheorem-divtype-full-implies-symmaverage}If $\alpha\in-\partial t(\beta)$
and $\dim_{H}\left(\pi_{\Phi}\left(\Lambda_{\mathrm{r}}\left(N\right)\cap\mathcal{F}\left(\alpha,\Phi,\psi\right)\right)\right)=\dim_{H}\left(\pi_{\Phi}\left(\mathcal{F}\left(\alpha,\Phi,\psi\right)\right)\right)$,
 then \linebreak  $\left(N,\Phi,\psi\right)$ is symmetric on average
in $\beta$. 
\item \label{enu:maintheorem-divtype-symmaverage-implies-full}If $\alpha\in-\left(\partial t_{N}(\beta)\cap\Int\left(\partial t_{N}(\R)\right)\right)$
and $\left(N,\Phi,\psi\right)$ is symmetric on average in $\beta$,
then \linebreak $\dim_{H}\left(\pi_{\Phi}\left(\Lambda_{\mathrm{r}}\left(N\right)\cap\mathcal{F}\left(\alpha,\Phi,\psi\right)\right)\right)=\dim_{H}\left(\pi_{\Phi}\left(\mathcal{F}\left(\alpha,\Phi,\psi\right)\right)\right)$.
\end{enumerate}
\end{enumerate}
\end{thm}
\begin{rem}
\label{rem-finitelygenerated-has-finiteindex}If $N$ is finitely
generated, then $\left(N,\Phi\right)$ is of divergence type by Lemma
\ref{lem:regular-iff-divergencetype}. Hence, $\F_{d}/N$ is amenable
by Theorem \ref{thm:maintheorem-limitset} (\ref{enu:maintheorem-limitset-divtype-amenable}).
In fact, it is well-known that, if $N$ is a finitely generated normal
subgroup of $\F_{d}$, then $\F_{d}/N$ is finite. In particular,
we then have that $\Lambda_{\mathrm{r}}\left(N\right)=\Lambda_{\mathrm{r}}\left(\F_{d}\right)$. 
\end{rem}

\begin{rem}
It is also worth noting that the proof of (\ref{enu:maintheorem-limitset-0})
in Theorem \ref{thm:maintheorem-limitset} and \ref{thm:maintheorem}
makes use of a certain induced graph directed Markov system $\tilde{\Phi}$
(Definition \ref{N-induced-GDMS}). This system $\tilde{\Phi}$ is
regular (\cite[Section 4, page 78]{MR2003772}) if and only if $\left(N,\Phi\right)$
is of divergence type by Lemma \ref{lem:regular-iff-divergencetype}.
In particular, we have by Theorem \ref{thm:maintheorem-limitset}
(\ref{enu:maintheorem-limitset-divtype-amenable}) that, if $\F_{d}/N$
is non-amenable, then $\tilde{\Phi}$ is irregular.
\end{rem}
Before we proceed to discuss applications to Kleinian groups, let
us briefly point out a relation to the cogrowth of group presentations\foreignlanguage{english}{
}$\left\langle g_{1},\dots,g_{d}|r_{1},r_{2},\dots\right\rangle $
on $d\ge2$ generators\foreignlanguage{english}{ (}\cite{MR599539}\foreignlanguage{english}{,
}\cite{MR678175})\foreignlanguage{english}{. If the geometric potential
of $\Phi$ is constant, then it is easy to see that the cogrowth $\eta$
of }$\left\langle g_{1},\dots,g_{d}|r_{1},r_{2},\dots\right\rangle $\foreignlanguage{english}{
is given by $\delta_{N}/\delta$, where $N$ is the normal subgroup
generated by $r_{1},r_{2},\dots$. Since every constant potential
is symmetric with respect to $N$, the criterion for $\delta_{N}=\delta$
in Theorem \ref{thm:maintheorem-limitset} (\ref{enu:maintheorem-nonamenable-implies-gap})
and (\ref{enu:maintheorem-amenablesymmetric-fulldimension}) is the
well-known cogrowth criterion $\eta=1$ for amenability of $\F_{d}/N$,
}and the\foreignlanguage{english}{ lower bound in Theorem \ref{thm:maintheorem}
(\ref{enu:maintheorem-symmetriclowerbound}) corresponds to the lower
bound $\eta>1/2$ for the cogrowth (}\cite{MR599539}\foreignlanguage{english}{,
}\cite{MR678175})\foreignlanguage{english}{. }

\subsection{Related results for Kleinian groups}

Let us now state the analogous results for Kleinian groups, which
have motivated our results. We first recall some basic notations for
Kleinian groups. For references on limit sets of Kleinian groups and
the associated hyperbolic manifolds, we refer the reader to \cite{Beardon,MR959135,MR1041575,MR1638795,MR2191250}. 

For a Kleinian group $\Gamma$ acting on the Poincar\'{e} disc model
$\mathbb{D}:=\left\{ z\in\R^{n+1}:\Vert z\Vert<1\right\} $ of hyperbolic
$\left(n+1\right)$-space, $n\in\N$, the Poincar\'{e} series is,
for each $s\in\R$, given by $P\left(\Gamma,s\right):=\sum_{\gamma\in\Gamma}\e^{-sd\left(0,\gamma\left(0\right)\right)}$,
where $d$ denotes the hyperbolic metric on $\mathbb{D}$. The exponent
of convergence of $\Gamma$ is given by $\delta\left(\Gamma\right):=\inf\left\{ s\ge0:P\left(\Gamma,s\right)<\infty\right\} $.
It is well-known by a theorem of Bishop and Jones (\cite{MR1484767})
that the exponent of convergence of a non-elementary Kleinian group
$\Gamma$ is equal to the Hausdorff dimension of the (uniformly) radial
limit set of $\Gamma$, that is 
\begin{equation}
\delta\left(\Gamma\right)=\dim_{H}\left(L_{\mathrm{ur}}\left(\Gamma\right)\right)=\dim_{H}\left(L_{\mathrm{r}}\left(\Gamma\right)\right).\label{eq:bishop-jones-intro}
\end{equation}

Passing from a non-elementary Kleinian group $\Gamma$ to a normal
subgroup $N$ of $\Gamma$ gives rise to a normal covering of the
associated hyperbolic manifolds. Brooks proved in \cite{MR783536}
that if $\Gamma$ is convex cocompact and $\delta\left(\Gamma\right)>n/2$,
then 
\begin{equation}
\delta\left(\Gamma\right)=\delta\left(N\right)\mbox{ if and only if }\Gamma/N\mbox{ is amenable}.\label{eq:amenability-dichotomy-intro}
\end{equation}
A recent result of Stadlbauer (\cite{Stadlbauer11}) shows that the
amenability dichotomy in (\ref{eq:amenability-dichotomy-intro}) holds
for all essentially free Kleinian groups $\Gamma$ with arbitrary
exponent of convergence. 

A complementary result is due to Falk and Stratmann (\cite[Theorem 2]{MR2097162})
which states that for each non-trivial normal subgroup $N$ of a non-elementary
Kleinian group $\Gamma$, we have that 
\begin{equation}
\delta\left(N\right)\ge\delta\left(\Gamma\right)/2.\label{eq:lower-bound-intro}
\end{equation}
Roblin (\cite{MR2166367})  proved that strict inequality in (\ref{eq:lower-bound-intro})
holds if the Kleinian group $\Gamma$ is of divergence type, that
is $P\left(\Gamma,\delta\left(\Gamma\right)\right)=\infty$. Another
proof of this result was independently obtained by Bonfert-Taylor,
Matsuzaki and Taylor (\cite{Bonfert-Taylor2012}), if $\Gamma$ is
convex cocompact. A related result by Matsuzaki and Yabuki in \cite{MR2486788}
states that 
\begin{equation}
\mbox{if }P\left(N,\delta\left(N\right)\right)=\infty\mbox{ then }\delta\left(\Gamma\right)=\delta\left(N\right).\label{eq:matsuyabuki-intro}
\end{equation}

In \cite{Jaerisch12a}, the author used this result of Matsuzaki and
Yabuki to give a short new proof of the strict inequality in (\ref{eq:lower-bound-intro})
if $\Gamma$ is of divergence type.
\begin{rem*}
A related result is due to Rees (\cite{MR627791,MR661820}) which
shows the following for a non-trivial normal subgroup $N$ of a convex
cocompact geometrically finite Fuchsian groups $\Gamma$ such that
$\Gamma/N\simeq\Z^{d}$ for some $d\in\N$: The critical exponents
$\delta\left(N\right)$ and $\delta\left(\Gamma\right)$ coincide.
Moreover, we have that $N$ is of divergence type if and only if $d\le2$. 
\end{rem*}
\selectlanguage{english}%
We first observe that Theorem \ref{thm:maintheorem-limitset} (\ref{enu:maintheorem-limitset-0})
gives the analog statement to (\ref{eq:bishop-jones-intro}). Moreover,
the results given in Theorem \ref{thm:maintheorem-limitset} (\ref{enu:maintheorem-limitset-nonamenable-implies-gap-1})
and (\ref{enu:maintheorem-limitset-amenablesymmetric-fulldimension-1})
extend the amenability dichotomy in (\ref{eq:amenability-dichotomy-intro}).
The lower bound in Theorem \ref{thm:maintheorem-limitset} (\ref{enu:maintheorem-limitset-symmetriclowerbound-1})
corresponds to (\ref{eq:lower-bound-intro}) and strict inequality
holds because the finitely generated group $\F_{d}$ is of divergence
type by \foreignlanguage{british}{Lemma \ref{lem:regular-iff-divergencetype}}.
Theorem \ref{thm:maintheorem-limitset} (\ref{enu:maintheorem-limitset-divtype-amenable})
and (\ref{enu:maintheorem-limitset-divtype-full-symmaverage}) shed
new light on (\ref{eq:matsuyabuki-intro}). Finally, the results of
Theorem \ref{thm:maintheorem} show that similar results hold, if
the radial limit set is intersected with the level set of a multifractal
decomposition with respect to a \foreignlanguage{british}{H\"older
continuous potential. }

Let us also remark that the symmetry assumptions imposed on the graph
directed Markov system $\Phi$ in Theorem \ref{thm:maintheorem-limitset}
and \ref{thm:maintheorem} mimic the property that $d(0,g(0))=d(0,g^{-1}(0))$
for every isometry $g$ with respect to the hyperbolic metric $d$.
This relation will be further illustrated in the following subsection.

\subsubsection{Application to normal subgroups of Kleinian groups of Schottky type}

\selectlanguage{british}%
Let us briefly describe a class of Kleinian groups, for which our
results are directly applicable. Let $\Gamma=\left\langle \gamma_{1},\dots,\gamma_{d}\right\rangle $
denote a Kleinian group of Schottky type with $d\ge2$. Then $\Gamma$
is a free group generated by hyperbolic transformations $\gamma_{1},\dots,\gamma_{d}$
(\cite[Definition 5.2]{Jaerisch11a}). Let $N$ denote a non-trivial
normal subgroup of $\Gamma$. The radial limit set $L_{r}\left(N\right)$
can be described in terms of a conformal graph directed Markov system
$\Phi_{\Gamma}$ associated to $\Gamma$. \foreignlanguage{english}{More
precisely, it is shown in }\cite[Proposition 5.6]{Jaerisch11a} that
there is a conformal graph directed Markov system $\Phi_{\Gamma}$
such that, for each non-trivial normal subgroup $N$ of $\Gamma$,
\begin{equation}
L_{\mathrm{r}}\left(N\right)=\pi_{\Phi_{\Gamma}}\left(\Lambda_{\mathrm{r}}\left(N\right)\right)\quad\mbox{and}\quad L_{\mathrm{ur}}\left(N\right)=\pi_{\Phi_{\Gamma}}\left(\Lambda_{\mathrm{ur}}\left(N\right)\right).\label{eq:schottky-via-gdms}
\end{equation}
The limit set $L(\Gamma)$ is equal to the limit set of $\Phi_{\Gamma}$
and a symbolic representation is given by $\Sigma:=\left\{ \tau\in I^{\N}:\tau_{i}\neq\tau_{i+1}^{-1}\right\} $
where $I:=\left\{ \gamma_{1},\gamma_{1}^{-1},\dots,\gamma_{d},\gamma_{d}^{-1}\right\} $.
The geometric potential $\zeta$ of $\Phi_{\Gamma}$ has the property
that there exists $C>0$ such that for all $\omega\in\F_{d}$, 
\begin{equation}
C^{-1}\e^{S_{\omega}\zeta}\le\e^{-d(0,\omega(0))}\le C\e^{S_{\omega}\zeta}.\label{eq:geometric-potential-vs-hyperbolicmetric}
\end{equation}
We refer to (see \cite{MR662473}) for details. We now consider the
(inverse) Lyapunov spectrum of $L(\Gamma)$ and its restriction to
$L_{r}\left(N\right)$. More precisely, we define for each $\alpha\in\R$
the level sets 
\[
\mathcal{L}\left(\alpha\right):=\left\{ \tau\in L\left(\Gamma\right):\lim_{n\rightarrow\infty}\frac{n}{d\left(0,\tau_{1}\dots\tau_{n}\left(0\right)\right)}=\alpha\right\} \quad\mbox{and}\quad\mathcal{L}_{N}\left(\alpha\right):=\mathcal{L}\left(\alpha\right)\cap L_{r}\left(N\right).
\]
Set $\alpha_{-}:=\min\left\{ \alpha\in\R:\mathcal{L}\left(\alpha\right)\neq\emptyset\right\} $
and $\alpha_{+}:=\max\left\{ \alpha\in\R:\mathcal{L}\left(\alpha\right)\neq\emptyset\right\} $.
As a corollary of Theorem \ref{thm:maintheorem}, we obtain the following
for the Lyapunov spectrum. 
\begin{cor}
Let $\Gamma$ denote a Kleinian group of Schottky type and let $N$
denote a non-trivial normal subgroup of $\Gamma$. Then for each $\alpha\in\left(\alpha_{-},\alpha_{+}\right)$
we have that 
\begin{equation}
\frac{\dim_{H}\left(\mathcal{L}\left(\alpha\right)\right)}{2}<\dim_{H}\left(\mathcal{L}_{N}\left(\alpha\right)\right)\le\dim_{H}\left(\mathcal{L}\left(\alpha\right)\right).\label{eq:lyapunov-kleinian}
\end{equation}
The second inequality in (\ref{eq:lyapunov-kleinian}) is an equality
if and only if $\Gamma/N$ is amenable. Moreover, if there exists
$\beta\in\R$ such that $\sum_{\omega\in N}\e^{t_{N}\left(\beta\right)S_{\omega}\zeta-\beta\left|\omega\right|}=\infty$,
then $\Gamma/N$ is amenable.\end{cor}
\begin{proof}
Recall that by (\ref{eq:schottky-via-gdms}), we have $L_{\mathrm{r}}\left(N\right)=\pi_{\Phi_{\Gamma}}\left(\Lambda_{\mathrm{r}}\left(N\right)\right)$.
Furthermore, we observe that by (\ref{eq:geometric-potential-vs-hyperbolicmetric})
we have $\mathcal{L}\left(\alpha\right)=\pi_{\Phi_{\Gamma}}\left(\mathcal{F}\left(\alpha,\Phi_{\Gamma},-\1\right)\right)$
where $\1:\Sigma\rightarrow\left\{ 1\right\} $. Since $\pi_{\Phi_{\Gamma}}$
is one-to-one we have that $\mathcal{L}_{N}\left(\alpha\right)=\pi_{\Phi_{\Gamma}}\left(\mathcal{F}\left(\alpha,\Phi_{\Gamma},-\1\right)\cap\Lambda_{\mathrm{r}}\left(N\right)\right)$.
Using that $d(0,\omega(0))=d(0,\omega^{-1}(0))$ for each $\omega\in\F_{d}$,
we have $C\e^{S_{\omega}\zeta}\le C\e^{S_{\omega^{-1}}\zeta}$ by
(\ref{eq:geometric-potential-vs-hyperbolicmetric}). Thus, we have
that $\left(N,\Phi_{\Gamma},-\1\right)$ is symmetric. The corollary
now follows from Theorem \ref{thm:maintheorem}. 
\end{proof}
\selectlanguage{english}%

\subsection{Plan of the paper}

The results stated in Theorem \ref{thm:maintheorem-limitset} and
\ref{thm:maintheorem} are deduced from general results of the thermodynamic
formalism for group-extended Markov systems developed by Stadlbauer
(\cite{Stadlbauer11}) and by the author (\cite{JaerischDissertation11,Jaerisch12b,Jaerisch12c},
see also \foreignlanguage{british}{\cite{Jaerisch11a}}). These results
are given in Section \foreignlanguage{british}{\ref{sec:Group-Extended-Markov}}.
The general results stated in Proposition \ref{prop:weak-lower-deltahalf-bound},
which are used to derive (\ref{enu:maintheorem-limitset-symmetriclowerbound-1})
of Theorem \ref{thm:maintheorem-limitset} and\foreignlanguage{british}{
\ref{thm:maintheorem}}, are new in the context of group-extended
Markov systems.\foreignlanguage{british}{ To obtain the strict inequality
in }(\ref{enu:maintheorem-limitset-symmetriclowerbound-1}) of Theorem
\ref{thm:maintheorem-limitset} and\foreignlanguage{british}{ \ref{thm:maintheorem}}
and the results involving the divergence of the \foreignlanguage{british}{Poincar\'{e}
series }of the normal subgroup, we make use of a characterization
of recurrent potentials from \cite{Jaerisch12b}. 

\selectlanguage{british}%
In Proposition \ref{prop:t-vs-dimension} we develop a multifractal
formalism for the multifractal decomposition of the radial limit set
parametrised by a normal subgroup. We make use of an induced graph
directed Markov system (Definition \ref{N-induced-GDMS}) which is
generated by infinitely many maps if the normal subgroup is of infinite
index. A similar argument as in \cite{MundayKessStrat10} allows us
to relate the limit set of the induced graph directed Markov system
to the level sets $\mathcal{F}\left(\alpha,\Phi,\psi\right)$ (see
Proposition \ref{prop:t-vs-dimension}). We then use the methods from
\cite{JaerischKessebohmer:09} to establish the multifractal formalism.

\selectlanguage{english}%
For a graph directed Markov system $\Phi$ consisting of similarities,
the results stated in Theorem \ref{thm:maintheorem-limitset} (\ref{enu:maintheorem-nonamenable-implies-gap}),
(\ref{enu:maintheorem-amenablesymmetric-fulldimension}) and (\ref{enu:maintheorem-symmetriclowerbound})
are contained in \foreignlanguage{british}{\cite{Jaerisch11a}. In
the present paper, we have generalised the results essentially in
two ways: }Firstly, Theorem \ref{thm:maintheorem-limitset}\foreignlanguage{british}{
applies to arbitrary conformal graph directed Markov systems and secondly,
Theorem \ref{thm:maintheorem} allows us to investigate intersections
of the radial limit set parametrised by a normal subgroup with level
sets of a multifractal decomposition with respect to a H\"older continuous
potential. Moreover, the result stated in Theorem \ref{thm:maintheorem-limitset}
(\ref{enu:maintheorem-limitset-divtype-full-symmaverage}) is new
even if $\Phi$ consists of  similarities.}

The outline of this paper is as follows. In Section \ref{sec:Thermodynamic-Formalism-for},
we collect the necessary preliminaries on the symbolic thermodynamic
formalism for Markov shifts. In Section \ref{sec:Graph-Directed-Markov},
we give the definition of conformal graph directed Markov systems
associated to free groups and their radial limit sets, and we develop
a multifractal formalism in this context. In Section \ref{sec:Group-Extended-Markov}
we give results on amenability and recurrence for group-extended Markov
systems, from which we deduce our main results in Section \ref{sec:Proof}. 

\selectlanguage{british}%

\section{Thermodynamic Formalism for Markov shifts\label{sec:Thermodynamic-Formalism-for}}

Throughout, the state space of the thermodynamic formalism will be
a \emph{Markov shift} $\Sigma$ given by 
\[
\Sigma:=\left\{ \tau:=\left(\tau_{1},\tau_{2},\ldots\right)\in I^{\N}:\,\, a\left(\tau_{i},\tau_{i+1}\right)=1\mbox{ for all }i\in\N\right\} ,
\]
where $I$ denotes a finite or countable \emph{alphabet}, the matrix
$A=\left(a(i,j)\right)\in\left\{ 0,1\right\} ^{I\times I}$ is the
\emph{incidence matrix} and the \emph{left shift map} $\sigma:\Sigma\rightarrow\Sigma$
is given by $\sigma\left(\tau_{1},\tau_{2},\dots\right):=\left(\tau_{2},\tau_{3},\dots\right)$,
for each $\tau\in\Sigma$. We denote by $\Sigma^{n}:=\left\{ \omega\in I^{n}:\,\, a\left(\omega_{i},\omega_{i+1}\right)=1,1\le i\le n-1\right\} $
the set of \emph{$A$-admissible words }of length $n\in\mathbb{N}$.
The set of $A$-admissible words of arbitrary length is given by $\Sigma^{*}:=\bigcup_{n\in\N}\Sigma^{n}$.
We define the \emph{word length function} $\left|\cdot\right|:\,\Sigma^{*}\cup\Sigma\rightarrow\N\cup\left\{ \infty\right\} $,
where for $\omega\in\Sigma^{*}$ we set $\left|\omega\right|$ to
be the unique $n\in\N$ such that $\omega\in\Sigma^{n}$ and for $\omega\in\Sigma$
we set $\left|\omega\right|:=\infty$. For each $\omega\in\Sigma^{*}\cup\Sigma$
and $n\in\N$ with $1\le n\le\left|\omega\right|$, we define $\omega_{|n}:=\left(\omega_{1},\dots,\omega_{n}\right)$.
For $\tau,\tau'\in\Sigma$, we let $\tau\wedge\tau':=\tau_{|l}$,
where $l:=\sup\bigl\{ n\in\N:\tau_{|n}=\tau'_{|n}\bigr\}$. For $n\in\N_{0}$
and $\omega\in\Sigma^{n}$, the \emph{cylinder set $\left[\omega\right]$}
is given by $\left[\omega\right]:=\left\{ \tau\in\Sigma:\tau_{|n}=\omega\right\} $. 

If $\Sigma$ is the Markov shift with alphabet $I$ whose incidence
matrix consists entirely of $1$s, then we have that $\Sigma=I^{\N}$
and $\Sigma^{n}=I^{n}$, for all $n\in\N$. Then we set $I^{*}:=\Sigma^{*}$.
For $\omega,\omega'\in I^{*}$ we denote by $\omega\omega'\in I^{*}$
the \emph{concatenation} of $\omega$ and $\omega'$, which is defined
by $\omega\omega':=\bigl(\omega_{1},\dots,\omega_{\left|\omega\right|},\omega'_{1},\dots,\omega'_{\left|\omega'\right|}\bigr)$
for $\omega,\omega'\in I^{*}$. Note that $I^{*}$ forms a semigroup
with respect to the concatenation operation. The semigroup $I^{*}$
is the free semigroup generated by $I$ and satisfies the universal
property that, for each semigroup $S$ and for every map $u:I\rightarrow S$,
there exists a unique semigroup homomorphism $\hat{u}:I^{*}\rightarrow S$
such that $\hat{u}\left(i\right)=u\left(i\right)$, for all $i\in I$. 

We equip $I^{\N}$ with the product topology of the discrete topology
on $I$. The Markov shift $\Sigma\subset I^{\N}$ is equipped with
the subspace topology. A countable basis of this topology on $\Sigma$
is given by the cylinder sets $\left\{ \left[\omega\right]:\omega\in\Sigma^{*}\right\} $.
We will make use of the following metrics generating the product topology
on $\Sigma$. For $\alpha>0$ fixed, we define the metric $d_{\alpha}:\Sigma\times\Sigma\rightarrow\R$
on $\Sigma$ given by 
\[
d_{\alpha}\left(\tau,\tau'\right):=\e^{-\alpha\left|\tau\wedge\tau'\right|},\mbox{ for all }\tau,\tau'\in\Sigma.
\]

For a function $\varphi:\Sigma\rightarrow\R$ and $n\in\N_{0}$ we
use the notation $S_{n}\varphi:\Sigma\rightarrow\R$ to denote the
\emph{ergodic sum} of $\varphi$ with respect to $\sigma$, in other
words, $S_{n}\varphi:=\sum_{i=0}^{n-1}\varphi\circ\sigma^{i}$. 

We say that a function $\varphi:\Sigma\rightarrow\R$ is \emph{$\alpha$-H\"older
continuous}, for some $\alpha>0$, if 
\[
V_{\alpha}\left(\varphi\right):=\sup_{n\ge1}\left\{ V_{\alpha,n}\left(\varphi\right)\right\} <\infty,
\]
where for each $n\in\N$ we let 
\[
V_{\alpha,n}\left(\varphi\right):=\sup\left\{ \frac{\left|\varphi\left(\tau\right)-\varphi\left(\tau'\right)\right|}{d_{\alpha}\left(\tau,\tau'\right)}:\tau,\tau'\in\Sigma,\left|\tau\wedge\tau'\right|\ge n\right\} .
\]
The function \emph{$\varphi$ }is\emph{ H\"older continuous} if there
exists $\alpha>0$ such that $\varphi$ is $\alpha$-H\"older continuous. 

\selectlanguage{english}%
The following fact is well-known.
\begin{fact}
[\cite{MR2003772}, Lemma 2.3.1] \label{fact-bounded-distortion-property}If
$\varphi:\Sigma\rightarrow\R$ is H\foreignlanguage{british}{\"o}lder
continuous, then there exists a constant $C_{\varphi}>0$ such that,
for all $\omega\in\Sigma^{*}$ and $\tau,\tau'\in\left[\omega\right]$,
we have 
\[
\left|S_{\left|\omega\right|}\varphi\left(\tau\right)-S_{\left|\omega\right|}\varphi\left(\tau'\right)\right|\le C_{\varphi}.
\]

\end{fact}
\selectlanguage{british}%
We will make use of the following notion of pressure introduced in
\cite[Definition 1.1]{JaerischKessebohmer10}.
\begin{defn}
[Induced topological pressure]\label{def:induced-topological-pressure}For
$\varphi,\Delta:\Sigma\rightarrow\R$ with $\Delta\ge0$, and $\mathcal{C}\subset\Sigma^{*}$
we define for $\eta>0$ the $\Delta$\emph{-induced pressure of} $\varphi$
(with respect to $\mathcal{C}$) by 
\[
\mathcal{P}_{\Delta}\left(\varphi,\mathcal{C}\right):=\limsup_{T\rightarrow\infty}\frac{1}{T}\log\sum_{{\omega\in\mathcal{C}\atop T-\eta<S_{\omega}\Delta\le T}}\e^{S_{\omega}\varphi},
\]
which takes values in $\R\cup\left\{ \pm\infty\right\} $. In here,
we set $S_{\omega}\varphi:=\sup_{\tau\in\left[\omega\right]}S_{\left|\omega\right|}\varphi\left(\tau\right)$. \end{defn}
\begin{rem*}
It was shown in \cite[Theorem 2.4]{JaerischKessebohmer10} that the
definition of $\mathcal{P}_{\Delta}\left(\varphi,\mathcal{C}\right)$
is in fact independent of the choice of $\eta>0$. For this reason
we do not refer to $\eta>0$ in the definition of the induced pressure. \end{rem*}
\begin{notation*}
If $\Delta$ and/or $\mathcal{C}$ is left out in the notation of
induced pressure, then we tacitly assume that $\Delta=1$ and/or $\mathcal{C}=\Sigma^{*}$.
That is, $\mathcal{P}\left(\varphi\right):=\mathcal{P}_{1}\left(\varphi,\Sigma^{*}\right)$. 
\end{notation*}
We will make use of the following mixing properties for a Markov shift
\foreignlanguage{english}{$\Sigma$ with alphabet $I$.}
\begin{itemize}
\item $\Sigma$ is \emph{irreducible} if, for all $i,j\in I$, there exists
$\omega\in\Sigma^{*}$ such that $i\omega j\in\Sigma^{*}$.
\item $\Sigma$ is \emph{topologically mixing} if, for all $i,j\in I$,
there exists $n_{0}\in\N$ with the property that, for all $n\ge n_{0}$,
there exists $\omega\in\Sigma^{n}$ such that $i\omega j\in\Sigma^{*}$.
\item $\Sigma$ is \emph{finitely irreducible} if there exists a finite
set $F\subset\Sigma^{*}$ with the property that, for all $i,j\in I$,
there exists $\omega\in F$ such that $i\omega j\in\Sigma^{*}$.  
\item $\Sigma$ is \emph{finitely primitive} if there exists $l\in\N$ and
a finite set $F\subset\Sigma^{l}$ with the property that, for all
$i,j\in I$, there exists $\omega\in F$ such that $i\omega j\in\Sigma^{*}$.
 \end{itemize}
\begin{rem*}
Note that $\Sigma$ is finitely primitive if and only if $\Sigma$
is topologically mixing and if $\Sigma$ satisfies the big images
and preimages  property (see \cite{MR1955261}). 
\end{rem*}
The following fact is taken from \cite[Remark 2.11, Remark 2.7]{JaerischKessebohmer10}.
\begin{fact}
\label{fac:criticalexponents-via-pressure}Let $\Sigma$ be finitely
irreducible, $\mathcal{C}\subset\Sigma^{*}$ and let $\varphi,\Delta:\Sigma\rightarrow\R$
with $\Delta\ge c>0$ for some $c>0$. Then we have 
\[
\mathcal{P}_{\Delta}\left(\varphi,\mathcal{C}\right)=\inf\left\{ u\in\R:\mathcal{P}\left(\varphi-u\Delta,\mathcal{C}\right)\le0\right\} =\inf\left\{ u\in\R:\sum_{\omega\in\mathcal{C}}\e^{S_{\omega}\left(\varphi-u\Delta\right)}<\infty\right\} .
\]
If additionally $\card\left(I\right)<\infty$, then $\mathcal{P}_{\Delta}\left(\varphi,\mathcal{C}\right)$
is the unique $u\in\R$ such that $\mathcal{P}\left(\varphi-u\Delta,\mathcal{C}\right)=0$. 
\end{fact}
\selectlanguage{english}%
The following notion of a Gibbs measure is fundamental for the thermodynamic
formalism (cf. \cite{MR0289084}, \cite{bowenequilibriumMR0442989}).
\begin{defn}
[Gibbs measure]\label{gibbs-measure}Let $\varphi:\Sigma\rightarrow\R$
be  continuous. We say that a Borel probability measure \emph{$\mu$
is a Gibbs measure for $\varphi$ }if there exists a constant $C_{\mu}>0$
such that 
\begin{equation}
C_{\mu}^{-1}\le\frac{\mu\left(\left[\omega\right]\right)}{\e^{S_{\left|\omega\right|}\varphi\left(\tau\right)-\left|\omega\right|\mathcal{P}\left(\varphi\right)}}\le C_{\mu},\mbox{ for all }\omega\in\Sigma^{*}\mbox{ and }\tau\in\left[\omega\right].\label{eq:gibbs-equation}
\end{equation}

\end{defn}
The following theorem is proved in \cite[Section 2]{MR2003772}. 
\begin{thm}
[Existence of Gibbs measures]\label{thm:existence-of-gibbs-measures}Let
$\Sigma$ be finitely irreducible and let $\varphi:\Sigma\rightarrow\R$
be H\"older continuous with $\mathcal{P}\left(\varphi\right)<\infty$.
Then there exists a unique $\sigma$-invariant Gibbs measure for $\varphi$. 
\end{thm}
\selectlanguage{british}%

\section{Graph Directed Markov Systems associated to Free Groups\label{sec:Graph-Directed-Markov}}

In this section first recall the definition of a conformal graph directed
Markov system (GDMS) introduced by Mauldin and \foreignlanguage{english}{Urba\'nski}
(\cite{MR2003772}). Then we give the definition of a GDMS associated
to a free group and the radial limit set.

\subsection{Preliminaries\label{sub:Preliminaries}}
\begin{defn}
\label{def-gdms}A \emph{graph directed Markov system (GDMS)} $\Phi=\left(V,\left(X_{v}\right)_{v\in V},E,i,t,\left(\phi_{e}\right)_{e\in E},A\right)$
consists of a finite vertex set $V$, a family of nonempty compact
metric spaces $\left(X_{v}\right)_{v\in V}$, a countable edge set
$E$, maps $i,t:E\rightarrow V$, a family of injective contractions
$\phi_{e}:X_{t\left(e\right)}\rightarrow X_{i\left(e\right)}$ with
Lipschitz constants bounded by some $0<s<1$, and an edge incidence
matrix $A\in\left\{ 0,1\right\} ^{E\times E}$, which has the property
that $a\left(e,f\right)=1$ implies $t\left(e\right)=i\left(f\right)$,
for all $e,f\in E$. The \emph{coding map }of $\Phi$ is given by
\[
\pi_{\Phi}:\Sigma_{\Phi}\rightarrow\oplus_{v\in V}X_{v},\,\,\mbox{ such that }\bigcap_{n\in\N}\phi_{\tau_{|n}}\left(X_{t\left(\tau_{n}\right)}\right)=\left\{ \pi_{\Phi}\left(\tau\right)\right\} ,\mbox{ for each }\tau\in\Sigma_{\Phi},
\]
where $\oplus_{v\in V}X_{v}$ denotes the disjoint union of the sets
$X_{v}$ and $\Sigma_{\Phi}$ denotes the Markov shift with alphabet
set $E$ and incidence matrix $A$. The \emph{limit set of $\Phi$}
is defined by $J\left(\Phi\right):=\pi_{\Phi}\left(\Sigma_{\Phi}\right)$.
Further, we set $J^{*}\left(\Phi\right):=\bigcup_{F\subset E,\,\card\left(F\right)<\infty}\pi_{\Phi}\left(\Sigma_{\Phi}\cap F^{\N}\right)$. 
\end{defn}

\begin{defn}
[Conformal GDMS, \cite{MR2003772}]\label{def:cgdms}The GDMS $\Phi=\left(V,\left(X_{v}\right)_{v\in V},E,i,t,\left(\phi_{e}\right)_{e\in E},A\right)$
is called \emph{conformal }if the following conditions are satisfied.

\renewcommand{\theenumi}{\alph{enumi}}
\begin{enumerate}
\item \label{enu:cgdms-a-phasespace}For $v\in V$, the \emph{phase space}
$X_{v}$ is a compact connected subset of the Euclidean space $\left(\R^{D},\Vert\cdot\Vert\right)$,
for some $D\geq1$, such that $X_{v}$ is equal to the closure of
its interior, that is $X_{v}=\overline{\Int(X_{v})}$. 
\item \textit{\emph{\label{enu:cgdms-b-osc}(}}\textit{Open set condition
}\textit{\emph{(OSC))}} For all $a,b\in E$ with $a\ne b$, we have
that 
\[
\phi_{a}\left(\Int(X_{t\left(a\right)})\right)\cap\phi_{b}\left(\Int(X_{t\left(b\right)})\right)=\emptyset.
\]

\item \label{enu:cgdms-c-conformalextension}For each vertex $v\in V$ there
exists an open connected set $W_{v}\supset X_{v}$ such that the map
$\phi_{e}$ extends to a $C^{1}$ conformal diffeomorphism of $W_{v}$
into $W_{i\left(e\right)}$, for every $e\in E$ with $t\left(e\right)=v$. 
\item \textit{\emph{\label{enu:cgdms-d-coneproperty}(}}\textit{Cone property}\textit{\emph{)}}
There exist $l>0$ and $0<\gamma<\pi/2$ such that, for each $v\in V$
and $x\in X_{v}\subset\R^{D}$ there exists an open cone $\Con(x,\gamma,l)\subset\Int(X_{v})$
with vertex $x$, central angle of measure $\gamma$ and altitude
$l$. 
\item \label{enu:cgdms-e-hoelderderivative}There are two constants $L\geq1$
and $\alpha>0$ such that for each $e\in E$ and $x,y\in X_{t\left(e\right)}$
we have 
\[
\left|\left|\phi_{e}'(y)\right|-\left|\phi_{e}'(x)\right|\right|\leq L\inf_{u\in W_{t\left(e\right)}}\left|\phi_{e}'\left(u\right)\right|\Vert y-x\Vert^{\alpha},
\]
where $\vert\cdot\vert$ refers to the operator norm of a bounded
linear operator on $\left(\R^{D},\Vert\cdot\Vert\right)$. 
\end{enumerate}
\end{defn}
\begin{lem}
[\cite{MR2003772}, Lemma 4.2.2] \label{lem:geometricpotential-is-hoelder}If
$\Phi$ is a conformal GDMS, then for all $\omega\in\Sigma_{\Phi}^{*}$
and for all $x,y\in W_{t\left(\omega\right)}$, we have 
\begin{equation}
\left|\log\left|\phi_{\omega}'(x)\right|-\log\left|\phi_{\omega}'(y)\right|\right|\le\frac{L}{1-s}\Vert x-y\Vert^{\alpha}.\label{eq:log-varphi-prime-hoelder}
\end{equation}
 \end{lem}
\begin{defn}
For a GDMS $\Phi$ satisfying (\ref{enu:cgdms-a-phasespace}) and
(\ref{enu:cgdms-c-conformalextension}) of Definition \ref{def:cgdms},
the\emph{ geometric potential $\zeta:\Sigma_{\Phi}\rightarrow\R^{-}$}
of $\Phi$ is given by 
\[
\zeta\left(\tau\right):=\log\left|\phi_{\tau_{1}}'\left(\pi_{\Phi}\left(\sigma\left(\tau\right)\right)\right)\right|,\mbox{ for all }\tau\in\Sigma_{\Phi}.
\]
The following fact follows from \cite[Proposition 4.2.7, Lemma 3.1.3]{MR2003772}
and Lemma \ref{lem:geometricpotential-is-hoelder}. \end{defn}
\begin{fact}
\label{fact-geometricpotential-is-hoelder}Suppose that a GDMS $\Phi$
satisfies (\ref{enu:cgdms-a-phasespace}) and (\ref{enu:cgdms-c-conformalextension})
of Definition \ref{def:cgdms} and that the inequality (\ref{eq:log-varphi-prime-hoelder})
in Lemma \ref{lem:geometricpotential-is-hoelder} holds. Then the
geometric potential $\zeta$ of $\Phi$  is H\"older continuous.
In particular, $\zeta$ is H\"older continuous if $\Phi$ is a conformal
GDMS.
\end{fact}
The following result is taken from \cite[Theorem 4.2.13]{MR2003772},
where finitely primitivity can be replaced by finitely irreducibility
(see also \cite[Theorem 3.7]{MR2413348}). The last equality in Theorem
\ref{thm:cgdms-bowen-formula} follows from Fact \ref{fac:criticalexponents-via-pressure}
because the  geometric potential $\zeta$ of $\Phi$ is bounded away
from zero by $-\log(s)$, where $s$ denotes the uniform bound for
the Lipschitz constants of the generators of $\Phi$ (see Definition
\ref{def-gdms}). 
\begin{thm}
[Generalised Bowen's formula]\label{thm:cgdms-bowen-formula}Let
$\Phi$ be a conformal GDMS with a finitely irreducible incidence
matrix $A$ and geometric potential $\zeta:\Sigma_{\Phi}\rightarrow\R^{-}$.
We then have that 
\[
\dim_{H}\left(J\left(\Phi\right)\right)=\dim_{H}\left(J^{*}\left(\Phi\right)\right)=\inf\left\{ s\in\R:\mathcal{P}\left(s\zeta\right)\le0\right\} =\inf\big\{\beta\in\R:\sum_{\omega\in\Sigma_{\Phi}^{*}}\e^{\beta S_{\omega}\zeta}<\infty\big\}.
\]
\end{thm}
\begin{rem}
\label{bowen-formula-without-4e}The generalised Bowen's formula also
holds if the GDMS $\Phi$ satisfies (\ref{enu:cgdms-a-phasespace})-(\ref{enu:cgdms-d-coneproperty})
of Definition \ref{def:cgdms} and the inequality (\ref{eq:log-varphi-prime-hoelder})
stated in Lemma \ref{lem:geometricpotential-is-hoelder}. To prove
this, we distinguish two cases. In the case $D\ge2$, it follows from
\cite[Proposition  4.2.1]{MR2003772} that, if $\Phi$ satisfies (\ref{enu:cgdms-a-phasespace})
and (\ref{enu:cgdms-c-conformalextension}) of Definition \ref{def:cgdms},
then $\Phi$ satisfies automatically (\ref{enu:cgdms-e-hoelderderivative})
with $\alpha=1$. If $D=1$ then a closer inspection of the proof
of \cite[Theorem 4.2.13]{MR2003772} shows that Definition \ref{def:cgdms}
(\ref{enu:cgdms-e-hoelderderivative}) is in fact only used to deduce
(\ref{eq:log-varphi-prime-hoelder}) of Lemma \ref{lem:geometricpotential-is-hoelder}
(cf. \cite[Lemma 2.2]{MR1387085}).
\end{rem}

\subsection{Radial limit sets\label{sub:Radial-limit-sets}}

Graph directed Markov systems associated to free groups an their radial
limit sets have been introduced in \cite[Definition 2.10]{Jaerisch11a}. 
\begin{defn}
[GDMS  associated to a free group, (uniformly) radial limit set]\label{def:gdms-associated-to-freegroup-and-radiallimitsets}Let
\emph{$\F_{d}=\langle g_{1},\dots,g_{d}\rangle$ }denote the free
group on $d\ge2$ generators. Let 
\[
I:=\left\{ g_{1},g_{1}^{-1},\dots,g_{d},g_{d}^{-1}\right\} \quad\mbox{and}\quad\Sigma:=\left\{ \tau\in I^{\N}:\tau_{i}\neq\tau_{i+1}^{-1}\right\} .
\]
Let $N$ be a non-trivial normal subgroup $N$ of $\F_{d}$. The \emph{symbolic
radial limit set} \emph{of $N$} and the \emph{symbolic uniformly
radial limit set of $N$} are given by 
\begin{eqnarray*}
\Lambda_{\mathrm{r}}\left(N\right) & := & \left\{ \tau\in\Sigma:\exists h\in\F_{d},\mbox{ such that }\tau_{1}\cdot\dots\cdot\tau_{n}\in hN\mbox{ for infinitely many }n\in\N\right\} \\
\mbox{and}\\
\Lambda_{\mathrm{ur}}\left(N\right) & := & \left\{ \tau\in\Sigma:\exists H\subset\F_{d}\mbox{ finite},\mbox{ such that }\tau_{1}\cdot\dots\cdot\tau_{n}\in HN\mbox{ for all }n\in\N\right\} .
\end{eqnarray*}
The GDMS $\Phi=\left(V,\left(X_{v}\right)_{v\in V},E,i,t,\left(\phi_{e}\right)_{e\in E},A\right)$
is \emph{associated to $\F_{d}=\langle g_{1},\dots,g_{d}\rangle$,}
$d\ge2$, if $V=\left\{ g_{1},g_{1}^{-1},\dots,g_{d},g_{d}^{-1}\right\} $,
$E=\left\{ \left(v,w\right)\in V^{2}:v\neq w^{-1}\right\} $, $i,t:E\rightarrow V$
are given by $i\left(v,w\right)=v$ and $t\left(v,w\right)=w$ and
the incidence matrix $A\in\left\{ 0,1\right\} ^{E\times E}$ satisfies
$a\left(e,f\right)=1$ if and only if $t\left(e\right)=i\left(f\right)$,
for all $e,f\in E$.  We will tacitly apply the canonical bijection
between the spaces 
\[
\Sigma_{\Phi}=\left\{ \left(\left(v_{1},v_{2}\right),\left(v_{2},v_{3}\right),\dots\right)\in\left(V\times V\right)^{\N}:v_{i}\neq v_{i+1}^{-1}\right\} \mbox{ and }\Sigma=\left\{ \left(v_{1},v_{2},\dots\right)\in V^{\N}:v_{i}\neq v_{i+1}^{-1}\right\} .
\]
We call $\pi_{\Phi}\left(\Lambda_{\mathrm{r}}\left(N\right)\right)$
and $\pi_{\Phi}\left(\Lambda_{\mathrm{ur}}\left(N\right)\right)$
the\emph{ radial} \emph{limit set} \emph{of $N$} \emph{with respect
to $\Phi$} and the \emph{uniformly radial limit set }of $N$ \emph{with
respect to $\Phi$. }
\end{defn}

\subsection{The induced GDMS\label{sub:The-induced-GDMS}}

In order to investigate the radial limit set of a normal subgroup
$N$ of $\F_{d}$ with respect to a GDMS $\Phi$ associated to $\F_{d}$,
we introduce an induced GDMS $\tilde{\Phi}$ whose edge set consists
of first return loops in the Cayley graph of $\F_{d}/N$. 
\begin{defn}
\label{N-induced-GDMS}Let $\Phi=\left(V,\left(X_{v}\right)_{v\in V},E,i,t,\left(\phi_{e}\right)_{e\in E},A\right)$
denote a conformal GDMS associated to $\F_{d}$ with $d\ge2$, and
let $N$ denote a non-trivial normal subgroup of $\F_{d}$. The \emph{$N$-induced}
GDMS of $\Phi$ is given by $\tilde{\Phi}:=\left(V,\left(X_{v}\right)_{v\in V},\tilde{E},\tilde{i},\tilde{t},\left(\tilde{\phi}_{\omega}\right)_{\omega\in\tilde{E}},\tilde{A}\right)$,
where the edge set $\tilde{E}$ is given by 
\[
\tilde{E}:=\left\{ \omega\in\Sigma_{\Phi}^{*}:i\left(\omega_{1}\right)\cdot\dots\cdot i\left(\omega_{\left|\omega\right|}\right)\in N\mbox{ and }i\left(\omega_{1}\right)\cdot\dots\cdot i\left(\omega_{k}\right)\notin N\mbox{ for }1\le k<\left|\omega\right|\right\} ,
\]
and the maps $\tilde{i},\tilde{t}:\tilde{E}\rightarrow V$ are \foreignlanguage{english}{given
by} $\tilde{i}\left(\omega\right):=i\left(\omega_{1}\right)$ and
\foreignlanguage{english}{$\tilde{t}\left(\omega\right):=t\left(\omega_{\left|\omega\right|}\right)$,
for each $\omega\in\tilde{E}$. Further,} the incidence matrix \foreignlanguage{english}{$\tilde{A}\in\left\{ 0,1\right\} ^{\tilde{E}\times\tilde{E}}$
is given by} $\tilde{a}\left(\omega,\omega'\right):=a\left(\omega_{|\omega|},\omega'_{1}\right)$
and the contractions $\left(\tilde{\phi}_{\omega}\right)_{\omega\in\tilde{E}}$
are defined by $\tilde{\phi}_{\omega}:=\phi_{\omega_{1}}\circ\dots\circ\phi_{\omega_{|\omega|}}$,
for each $\omega\in\tilde{E}$. \end{defn}
\begin{notation}
For the $N$-induced GDMS $\tilde{\Phi}$, there are canonical embeddings
$\Sigma_{\tilde{\Phi}}^{*}\hookrightarrow\Sigma_{\Phi}^{*}$ and $\Sigma_{\tilde{\Phi}}\hookrightarrow\Sigma_{\Phi}$,
which we will both denote by $\iota$. It will always be clear which
map is in use. \end{notation}
\begin{defn}
For a function $f:\Sigma_{\Phi}\rightarrow\R$, its \emph{induced
version} $\tilde{f}:\Sigma_{\tilde{\Phi}}\rightarrow\R$ is given
by $\tilde{f}\left(\tilde{\tau}\right):=S_{|\iota\left(\tilde{\tau}_{1}\right)|}f(\iota\left(\tau\right))$,
for each $\tilde{\tau}=\left(\tilde{\tau}_{1},\tilde{\tau}_{2},\dots\right)\in\Sigma_{\tilde{\Phi}}$. 
\end{defn}
The proof of the following lemma is straightforward and therefore
omitted.
\begin{lem}
\label{lem:induced-gdms-facts}Let $\Phi$ denote a conformal GDMS
associated to $\F_{d}$ with $d\ge2$. Let $N$ denote a non-trivial
normal subgroup of $\F_{d}$, and let $\tilde{\Phi}$ denote the $N$-induced
GDMS of $\Phi$. We then have the following.
\begin{enumerate}
\item \label{enu:induced-gdms-finitelyirreducible}The incidence matrix
$\tilde{A}$ of $\tilde{\Phi}$ is finitely irreducible. 
\item \label{enu:induced-gdms-codingmaps}For the coding maps $\pi_{\tilde{\Phi}}:\Sigma_{\tilde{\Phi}}\rightarrow J\left(\tilde{\Phi}\right)$
and $\pi_{\Phi}:\Sigma_{\Phi}\rightarrow J\left(\Phi\right)$, we
have $\pi_{\tilde{\Phi}}\left(\tilde{\tau}\right)=\pi_{\Phi}\left(\iota\left(\tilde{\tau}\right)\right)$
for each $\tilde{\tau}\in\Sigma_{\tilde{\Phi}}$.
\item \label{enu:induced-gdms-ergodicsums}The geometric potential $\tilde{\zeta}:\Sigma_{\tilde{\Phi}}\rightarrow\R$
of $\tilde{\Phi}$ is the induced version of the geometric potential
$\zeta:\Sigma_{\Phi}\rightarrow\R$ of $\Phi$. 
\item \label{enu:induced-gdms-hoelderdistortion}Let $f:\Sigma_{\Phi}\rightarrow\R$
be H\"older continuous. Then the induced version $\tilde{f}:\Sigma_{\tilde{\Phi}}\rightarrow\R$
is H\"older continuous and there exists a constant $C_{f}>0$ such
that $S_{\iota\left(\tilde{\omega}\right)}f-C_{f}\le S_{\tilde{\omega}}\tilde{f}\le S_{\iota\left(\tilde{\omega}\right)}f$,
\linebreak for all $\tilde{\omega}\in\Sigma_{\tilde{\Phi}}^{*}$. 
\end{enumerate}
\end{lem}
The next proposition provides a version of Bowen's formula for the
Hausdorff dimension of the radial limit set of a normal subgroup $N$
of $\F_{d}$ with respect to a conformal GDMS $\Phi$ associated to
$\F_{d}$. This extends \cite[Proposition 6.2.8]{JaerischDissertation11}
and \cite[Proposition 1.3]{Jaerisch11a}. Moreover, we establish a
multifractal formalism in this context. 
\begin{prop}
\label{prop:t-vs-dimension}Let $\Phi$ denote a conformal GDMS associated
to \emph{$\F_{d}$ with} $d\ge2$, and let $N$ denote a non-trivial
normal subgroup of $\F_{d}$. Let $\psi:\Sigma_{\Phi}\rightarrow\R$
be H\"older continuous and let $t_{N}$ denote the free energy function
of $\left(N,\Phi,\psi\right)$. Then we have the following. 
\selectlanguage{english}%
\begin{enumerate}
\item \textup{\label{enu:t0-dimension-limitset}$\dim_{H}\left(\pi_{\Phi}\left(\Lambda_{\mathrm{ur}}\left(N\right)\right)\right)=\dim_{H}\left(\pi_{\Phi}\left(\Lambda_{\mathrm{r}}\left(N\right)\right)\right)=\delta_{N}$. }
\selectlanguage{british}%
\item \label{enu:multifractal-upperbound}$\dim_{H}\left(\pi_{\Phi}\left(\Lambda_{\mathrm{r}}\left(N\right)\cap\mathcal{F}\left(\alpha,\Phi,\psi\right)\right)\right)\le\max\left\{ -t_{N}^{*}\left(-\alpha\right),0\right\} $,
for each $\alpha\in\R$. 
\item \label{enu:multifractal-formalism}$\dim_{H}\left(\pi_{\Phi}\left(\Lambda_{\mathrm{ur}}\left(N\right)\cap\mathcal{F}\left(\alpha,\Phi,\psi\right)\right)\right)=\dim_{H}\left(\pi_{\Phi}\left(\Lambda_{\mathrm{r}}\left(N\right)\cap\mathcal{F}\left(\alpha,\Phi,\psi\right)\right)\right)=-t_{N}^{*}\left(-\alpha\right)>0$,
for each $\alpha\in-\Int\left(\partial t_{N}\left(\R\right)\right)$.
\end{enumerate}
\end{prop}
\begin{proof}
Let $\tilde{\Phi}$ denote the $N$-induced GDMS of $\Phi$. First,
we relate the limit set of $\tilde{\Phi}$ to the radial limit set
of $N$ with respect to $\Phi$. Using Lemma \ref{lem:induced-gdms-facts}
(\ref{enu:induced-gdms-codingmaps}), it is straightforward to verify
that 
\begin{equation}
J^{*}\left(\tilde{\Phi}\right)\subset\pi_{\Phi}\left(\Lambda_{\mathrm{ur}}\left(N\right)\right)\subset\pi_{\Phi}\left(\Lambda_{\mathrm{r}}\left(N\right)\right)\subset J\left(\tilde{\Phi}\right)\cup\bigcup_{\eta\in\Sigma_{\Phi}^{*},\tilde{\tau}\in\Sigma_{\tilde{\Phi}}:\eta\iota\left(\tilde{\tau}\right)\in\Sigma_{\Phi}}\phi_{\eta}\left(\pi_{\tilde{\Phi}}\left(\tilde{\tau}\right)\right).\label{eq:Lur-loop-gdms-1-1}
\end{equation}
Note that the right-hand side of (\ref{eq:Lur-loop-gdms-1-1})  is
a countable union of Lipschitz continuous images of \foreignlanguage{english}{$J\left(\tilde{\Phi}\right)$.
Since Lipschitz continuous maps do not increase Hausdorff dimension
and since Hausdorff dimension is stable under countable unions, we
obtain that 
\begin{equation}
\dim_{H}\left(J^{*}\left(\tilde{\Phi}\right)\right)\le\dim_{H}\left(\pi_{\Phi}\left(\Lambda_{\mathrm{ur}}\left(N\right)\right)\right)\le\dim_{H}\left(\pi_{\Phi}\left(\Lambda_{\mathrm{r}}\left(N\right)\right)\right)\le\dim_{H}\left(J\left(\tilde{\Phi}\right)\right).\label{eq:Lur-loop-gdems-1-2}
\end{equation}
}The GDMS $\tilde{\Phi}$ satisfies the conditions (\ref{enu:cgdms-a-phasespace})-(\ref{enu:cgdms-d-coneproperty})
in Definition \ref{def:cgdms}. Further, since $\bigl\{\tilde{\phi}_{\tilde{\omega}}:\tilde{\omega}\in\Sigma_{\tilde{\Phi}}^{*}\bigr\}$
is a subfamily of $\bigl\{\phi_{\omega}:\omega\in\Sigma_{\Phi}^{*}\bigr\}$,
it follows that $\tilde{\Phi}$ satisfies (\ref{eq:log-varphi-prime-hoelder})
of Lemma \ref{lem:geometricpotential-is-hoelder}. \foreignlanguage{english}{Moreover,
by Lemma }\ref{lem:induced-gdms-facts} (\ref{enu:induced-gdms-finitelyirreducible}),
the incidence matrix of $\tilde{\Phi}$ is finitely irreducible. Hence,
b\foreignlanguage{english}{y }Remark \ref{bowen-formula-without-4e}\foreignlanguage{english}{,}
the generalised Bowen's formula in Theorem \ref{thm:cgdms-bowen-formula}
and (\ref{eq:Lur-loop-gdems-1-2}) give that 
\begin{equation}
\dim_{H}\left(\pi_{\Phi}\left(\Lambda_{\mathrm{ur}}\left(N\right)\right)\right)=\dim_{H}\left(\pi_{\Phi}\left(\Lambda_{\mathrm{r}}\left(N\right)\right)\right)=\inf\biggl\{\beta\in\R:\sum_{\tilde{\omega}\in\Sigma_{\tilde{\Phi}}^{*}}\e^{\beta S_{\tilde{\omega}}\tilde{\zeta}}<\infty\biggr\}.\label{eq:Lur-loop-gdms-2-1}
\end{equation}
For each $\tilde{\omega}\in\Sigma_{\tilde{\Phi}}^{*}$ there exists
$n\in\N$ such that $\iota\left(\tilde{\omega}\right)\in\Sigma_{\Phi}^{n}$.
Write $\iota\left(\tilde{\omega}\right)=\left(\left(v_{1},w_{1}\right),\left(v_{2},w_{2}\right),\dots,\left(v_{n},w_{n}\right)\right)$.
By mapping each element $\left(\left(v_{1},w_{1}\right),\left(v_{2},w_{2}\right),\dots,\left(v_{n},w_{n}\right)\right)$
to $\left(v_{1}v_{2}\dots v_{n}\right)$, we obtain a $\left(2d-1\right)$-to-one
map from $\Sigma_{\tilde{\Phi}}^{*}$ onto $N\setminus\left\{ \id\right\} $.
Hence, by \foreignlanguage{english}{Lemma }\ref{lem:induced-gdms-facts}
(\ref{enu:induced-gdms-hoelderdistortion}), we see that 
\begin{equation}
\inf\Bigl\{\beta\in\R:\sum_{\tilde{\omega}\in\Sigma_{\tilde{\Phi}}^{*}}\e^{\beta S_{\tilde{\omega}}\tilde{\zeta}}<\infty\Bigr\}=\inf\Bigl\{\beta\in\R:\sum_{\omega\in N\setminus\left\{ \id\right\} }\e^{\beta S_{\omega}\zeta}<\infty\Bigr\}=\delta_{N}.\label{eq:t-N-via-series}
\end{equation}
Combining (\ref{eq:Lur-loop-gdms-2-1}) and (\ref{eq:t-N-via-series})
finishes the proof of (\ref{enu:t0-dimension-limitset}).

For the remaining part of the proof, we define for each $\alpha\in\R$
the symbolic level sets 
\begin{align*}
\tilde{\mathcal{F}}^{*}\left(\alpha\right) & :=\Bigl\{\tilde{\tau}=(\tilde{\tau}_{1},\tilde{\tau}_{2},\dots)\in\Sigma_{\tilde{\Phi}}:\lim_{k\rightarrow\infty}\frac{S_{(\tilde{\tau}_{1},\dots,\tilde{\tau}_{k})}\tilde{\psi}}{S_{(\tilde{\tau}_{1},\dots,\tilde{\tau}_{k})}\tilde{\zeta}}=\alpha\mbox{ and }\sup_{i\in\N}\left\{ \iota\left(\tilde{\tau}_{i}\right)\right\} <\infty\Bigr\}\\
 & \mbox{and }\\
\tilde{\mathcal{F}}\left(\alpha\right) & :=\Bigl\{\tilde{\tau}=(\tilde{\tau}_{1},\tilde{\tau}_{2},\dots)\in\Sigma_{\tilde{\Phi}}:\lim_{k\rightarrow\infty}\frac{S_{(\tilde{\tau}_{1},\dots,\tilde{\tau}_{k})}\tilde{\psi}}{S_{(\tilde{\tau}_{1},\dots,\tilde{\tau}_{k})}\tilde{\zeta}}=\alpha\Bigr\}.
\end{align*}
Furthermore, we define \foreignlanguage{english}{$\tilde{t}:\R\rightarrow\R\cup\left\{ \infty\right\} $
which is for each $\beta\in\R$ given by} 
\[
\tilde{t}\left(\beta\right)=\inf\left\{ u\in\R:\mathcal{P}\left(\beta\tilde{\psi}+u\tilde{\zeta}\right)\le0\right\} .
\]
By Fact \ref{fac:criticalexponents-via-pressure} we have that for
each $\beta\in\R$, 
\[
\tilde{t}\left(\beta\right)=\inf\Bigl\{ u\in\R:\sum_{\tilde{\omega}\in\Sigma_{\tilde{\Phi}}^{*}}\e^{\beta S_{\tilde{\omega}}\tilde{\psi}+uS_{\tilde{\omega}}\tilde{\zeta}}<\infty\Bigr\}.
\]
Following the proof of (\ref{eq:t-N-via-series}), we see that, for
each $\beta\in\R$, we have 
\begin{equation}
\tilde{t}\left(\beta\right)=t_{N}\left(\beta\right).\label{eq:t_N-is-inducedfreeenergy}
\end{equation}
To prove the upper bound for the Hausdorff dimension in (\ref{enu:multifractal-upperbound}),
first observe that, similarly as in (\ref{eq:Lur-loop-gdms-1-1}),
the set $\pi_{\Phi}\left(\Lambda_{\mathrm{r}}\left(N\right)\cap\mathcal{F}\left(\alpha,\Phi,\psi\right)\right)$
is contained in a countable union of Lipschitz continuous images of
$\pi_{\tilde{\Phi}}\left(\tilde{\mathcal{F}}\left(\alpha\right)\right)$.
Then by a standard covering argument (see e.g. \cite[Theorem 1.2]{JaerischKessebohmer:09})
we have $\dim_{H}\left(\tilde{\mathcal{F}}\left(\alpha\right)\right)\le\max\left\{ -\tilde{t}^{*}\left(-\alpha\right),0\right\} $,
which - in light of (\ref{eq:t_N-is-inducedfreeenergy}) - finishes
the proof of (\ref{enu:multifractal-upperbound}). 

To prove (\ref{enu:multifractal-formalism}), we first verify that
for each $\alpha\in\R$, 
\begin{equation}
\iota\left(\tilde{\mathcal{F}}^{*}\left(\alpha\right)\right)\subset\Lambda_{\mathrm{ur}}\left(N\right)\cap\mathcal{F}\left(\alpha,\Phi,\psi\right).\label{eq:induced-levelsets}
\end{equation}
\foreignlanguage{english}{Clearly, we have that $\iota\left(\tilde{\mathcal{F}}^{*}\left(\alpha\right)\right)\subset\Lambda_{\mathrm{ur}}\left(N\right)$,
so it remains to show that $\iota\left(\tilde{\mathcal{F}}^{*}\left(\alpha\right)\right)\subset\mathcal{F}\left(\alpha,\Phi,\psi\right)$.
The proof follows \cite[Proposition 4.3]{MundayKessStrat10}. Let
$\tilde{\tau}\in\tilde{\mathcal{F}}^{*}\left(\alpha\right)$ be given
and set $\tau:=\iota\left(\tilde{\tau}\right)$. Then there exists
$l\in\N$ such that $\left|\iota\left(\tilde{\tau}_{i}\right)\right|\le l$,
for each $i\in\N$. For each $n>l$, let 
\[
k\left(n\right):=\max\Bigl\{ k\in\N:\sum_{i=1}^{k}\iota\left(\left|\tilde{\tau}_{i}\right|\right)\le n\Bigr\}.
\]
Then there exists $r\left(n\right)\le l$ such that $n=\sum_{i=1}^{k\left(n\right)}\left|\iota\left(\tilde{\tau}_{i}\right)\right|+r\left(n\right)$.
For $f\in\left\{ \zeta,\psi\right\} $ we set 
\[
M_{f}=\max\left\{ S_{\omega}f:\omega\in\Sigma^{r}:1\le r\le l\right\} \quad\mbox{and}\quad m_{f}:=\min\left\{ S_{\omega}f:\omega\in\Sigma^{r}:1\le r\le l\right\} .
\]
Since $\zeta$ and $\psi$ are }H\"older continuous, we have by \foreignlanguage{english}{Lemma
}\ref{lem:induced-gdms-facts} (\ref{enu:induced-gdms-hoelderdistortion})
and Fact \ref{fact-bounded-distortion-property} that\foreignlanguage{english}{
\[
S_{(\tilde{\tau}_{1},\dots,\tilde{\tau}_{k\left(n\right)})}\tilde{f}+m_{f}-2C_{f}=S_{(\tau_{1},\dots,\tau_{n})}f=S_{(\tilde{\tau}_{1},\dots,\tilde{\tau}_{k\left(n\right)})}\tilde{f}+M_{f}+C_{f},
\]
which then gives 
\[
\frac{S_{(\tilde{\tau}_{1}\cdot\dots\cdot\tilde{\tau}_{k\left(n\right)})}\tilde{\psi}+m_{\psi}-2C_{\psi}}{S_{(\tilde{\tau}_{1}\cdot\dots\cdot\tilde{\tau}_{k\left(n\right)})}\tilde{\zeta}+M_{\zeta}+C_{\zeta}}\le\frac{S_{(\tau_{1},\dots,\tau_{n})}\psi}{S_{(\tau_{1},\dots,\tau_{n})}\zeta}\le\frac{S_{(\tilde{\tau}_{1}\cdot\dots\cdot\tilde{\tau}_{k\left(n\right)})}\tilde{\psi}+M_{\psi}+C_{\psi}}{S_{(\tilde{\tau}_{1}\cdot\dots\cdot\tilde{\tau}_{k\left(n\right)})}\tilde{\zeta}+m_{\zeta}-2C_{\zeta}}.
\]
Since $\tilde{\tau}\in\tilde{\mathcal{F}}^{*}\left(\alpha\right)$
and $\bigl|S_{(\tilde{\tau}_{1}\cdot\dots\cdot\tilde{\tau}_{k\left(n\right)})}\tilde{\zeta}\bigr|\rightarrow\infty$,
as $n\rightarrow\infty$, it follows that $\tau\in\mathcal{F}\left(\alpha,\Phi,\psi\right)$.
The proof of (\ref{eq:induced-levelsets}) is complete. By Lemma \ref{lem:induced-gdms-facts}
(\ref{enu:induced-gdms-codingmaps}) we conclude that $\pi_{\tilde{\Phi}}\left(\tilde{\mathcal{F}}^{*}\left(\alpha\right)\right)\subset\pi_{\Phi}\left(\Lambda_{\mathrm{ur}}\left(N\right)\cap\mathcal{F}\left(\alpha,\Phi,\psi\right)\right)$.
}Combining with the upper bound in (\ref{enu:multifractal-upperbound})
and (\ref{eq:t_N-is-inducedfreeenergy}), the proof will be completed,
if we have shown that 
\[
\dim_{H}\left(\pi_{\tilde{\Phi}}\left(\tilde{\mathcal{F}}^{*}\left(\alpha\right)\right)\right)=-\tilde{t}^{*}\left(-\alpha\right)>0.
\]
Since we assume that $\alpha\in-\Int(\partial\tilde{t}\left(\R\right))$,
a straightforward modification of \cite[Proof of Theorem 1.2]{JaerischKessebohmer:09}
shows that $\dim_{H}\left(\pi_{\tilde{\Phi}}\left(\tilde{\mathcal{F}}^{*}\left(\alpha\right)\right)\right)=\dim_{H}\left(\pi_{\tilde{\Phi}}\left(\tilde{\mathcal{F}}\left(\alpha\right)\right)\right)=-\tilde{t}^{*}\left(-\alpha\right)$.
The crucial step of this modification is to show that the function
$\tilde{t}$ satisfies the exhaustion principle, that is, $\tilde{t}\left(\beta\right)=\sup_{n\in\N}\tilde{t}_{n}\left(\beta\right)$,
where $\tilde{t}_{n}$ is for each $\beta\in\R$ given by 
\[
\tilde{t}_{n}\left(\beta\right):=\inf\Bigl\{ u\in\R:\sum_{k\in\N}\sum_{\left(\tilde{\omega}_{1},\dots,\tilde{\omega}_{k}\right)\in\Sigma_{\tilde{\Phi}}^{k}:\left|\iota\left(\tilde{\omega}_{i}\right)\right|\le n,1\le i\le k}\e^{\beta S_{\tilde{\omega}}\tilde{\psi}+uS_{\tilde{\omega}}\tilde{\zeta}}<\infty\Bigr\}.
\]
This can be verified similarly as in \cite[Example 1.6, Theorem 1.7]{JaerischKessebohmer:09}
by using that $\Sigma_{\tilde{\Phi}}$ is finitely irreducible (cf.
\cite[Theorem 2.1.5]{MR2003772}). Finally, to prove that $-\tilde{t}^{*}\left(-\alpha\right)>0$,
we observe that $-\tilde{t}^{*}\left(-\alpha\right)\ge-\tilde{t}_{n}^{*}\left(-\alpha\right)>0$
for all $n$ sufficiently large, which follows from the well-known
fact that  $-\tilde{t}_{n}^{*}$ is a non-negative, strictly concave
and real-analytic function. The proof is complete.
\end{proof}
Recall that the GDMS $\tilde{\Phi}$ is called \emph{regular} (\cite[Section 4, p.78]{MR2003772})
if there exists $u\in\R$ such that $\mathcal{P}\bigl(u\tilde{\zeta}\bigr)=0$. 
\begin{lem}
\label{lem:regular-iff-divergencetype}Let $\Phi$ denote a conformal
GDMS associated to \emph{$\F_{d}$ with} $d\ge2$, and let $N$ denote
a non-trivial normal subgroup of $\F_{d}$. Let $\psi:\Sigma_{\Phi}\rightarrow\R$
be H\"older continuous and let $\beta\in\R$. Then $\left(N,\Phi,\psi\right)$
is of divergence type in $\beta$ if and only if there exists $u\in\R$
such that $\mathcal{P}\bigl(\beta\tilde{\psi}+u\tilde{\zeta}\bigr)=0$.
Moreover, if $N$ is finitely generated then $\left(N,\Phi,\psi\right)$
is of divergence type in $\beta$. \end{lem}
\begin{proof}
Suppose that $\left(N,\Phi,\psi\right)$ is of divergence type in
$\beta$. Then we have $\mathcal{P}\bigl(\beta\tilde{\psi}+t_{N}\left(\beta\right)\tilde{\zeta}\bigr)\ge0$.
Further, by Fact \ref{fac:criticalexponents-via-pressure}, we have
that $t_{N}\left(\beta\right)=\inf\bigl\{ u\in\R:\mathcal{P}\bigl(\beta\tilde{\psi}+u\tilde{\zeta}\bigr)\le0\bigr\}$,
since $\Sigma_{\tilde{\Phi}}$ is finitely irreducible by Lemma\foreignlanguage{english}{
}\ref{lem:induced-gdms-facts} (\ref{enu:induced-gdms-finitelyirreducible}).
Using again that $\Sigma_{\tilde{\Phi}}$ is finitely irreducible,
it follows from \cite[Theorem 2.1.5]{MR2003772} that the map $u\mapsto\mathcal{P}\bigl(\beta\tilde{\psi}+u\tilde{\zeta}\bigr)\in\R\cup\left\{ \infty\right\} $
is the monotone limit of a sequence of continuous functions. Consequently,
the map $u\mapsto\mathcal{P}\bigl(\beta\tilde{\psi}+u\tilde{\zeta}\bigr)$
is lower semi-continuous, which then implies that  $\mathcal{P}\bigl(\beta\tilde{\psi}+t_{N}\left(\beta\right)\tilde{\zeta}\bigr)\le0$.
Hence, we have $\mathcal{P}\bigl(\beta\tilde{\psi}+t_{N}\left(\beta\right)\tilde{\zeta}\bigr)=0$. 

To prove the converse, suppose that $\mathcal{P}\bigl(\beta\tilde{\psi}+t_{N}\left(\beta\right)\tilde{\zeta}\bigr)=0$.
Since $\Sigma_{\tilde{\Phi}}$ is finitely irreducible, there exists
a Gibbs measure $\mu$ for the potential $\beta\tilde{\psi}+t_{N}\left(\beta\right)\tilde{\zeta}$
supported on \foreignlanguage{english}{$\Sigma_{\tilde{\Phi}}$} by
Theorem \ref{thm:existence-of-gibbs-measures}. Hence, we have that
$\sum_{n\in\N}\sum_{\tilde{\omega}\in\Sigma_{\tilde{\Phi}}^{n}}\mu\left(\left[\tilde{\omega}\right]\right)=\infty$.
Since $\mathcal{P}\bigl(\beta\tilde{\psi}+t_{N}\left(\beta\right)\tilde{\zeta}\bigr)=0$,
it follows from the Gibbs property of $\mu$ (\ref{eq:gibbs-equation})
that $\left(N,\Phi,\psi\right)$ is of divergence type in $\beta$. 

To finish the proof, suppose that $N$ is finitely generated. Then
the edge set of $\tilde{\Phi}$ is finite and hence, $\tilde{\Phi}$
is regular. In particular, we have that $\left(N,\Phi,\psi\right)$
is of divergence type in $\beta$. The proof is complete.
\end{proof}

\section{Group-extended Markov systems \label{sec:Group-Extended-Markov}}

Throughout this section, let $I$ denote a finite or countable alphabet
and let $I^{*}$ denote the free semigroup generated by $I$. Let
$G$ denote a countable group $G$ and let $\Psi:I^{*}\rightarrow G$
denote a semigroup homomorphism. The skew product dynamical system
$\sigma\rtimes\Psi:\Sigma\times G\rightarrow\Sigma\times G$, which
is given by 
\[
\left(\sigma\rtimes\Psi\right)\left(\tau,g\right):=\left(\sigma\left(\tau\right),g\Psi\left(\tau_{1}\right)\right),\mbox{ for all }\left(\tau,g\right)\in\Sigma\times G,
\]
is called a \emph{group-extended Markov system} (see also \cite[Section 4]{Jaerisch11a}).
Note that $\left(\Sigma\times G,\sigma\rtimes\Psi\right)$ is conjugated
to the Markov shift with state space 
\[
\left\{ \left(\left(\tau_{j},g_{j}\right)\right)\in\left(I\times G\right)^{\N}:\left(\tau_{j}\right)\in\Sigma,\forall i\in\N\, g_{i}\Psi\left(\tau_{i}\right)=g_{i+1}\right\} .
\]
Let $\pi_{1}:\Sigma\times G\rightarrow\Sigma$ denote the canonical
projection. 

We will make use of the following notions of symmetry for group-extended
Markov systems.
\begin{defn}
\label{def-asympt-symmetric}Let $\left(\Sigma\times G,\sigma\rtimes\Psi\right)$
denote an irreducible group-extended Markov system\emph{.} We say
that $\varphi$ is \emph{asymptotically symmetric }with respect to
$\Psi$ (\cite[Definition 3.14]{Jaerisch11a}) if there exist $n_{0}\in\N$
and sequences $\left(c_{n}\right)\in\left(\R^{+}\right)^{\N}$ and
$\left(N_{n}\right)\in\N^{\N}$ with $\lim_{n}\left(c_{n}\right)^{1/n}=1$
and $\lim_{n}n^{-1}N_{n}=0$, such that for each $g\in G$ and for
all $n\ge n_{0}$, 
\[
\sum_{\omega\in\Sigma^{n}\cap\Psi^{-1}(g)}\e^{S_{\omega}\varphi}\le c_{n}\sum_{\omega\in\Sigma^{*}\cap\Psi^{-1}(g^{-1}):n-N_{n}\le\left|\omega\right|\le n+N_{n}}\e^{S_{\omega}\varphi}.
\]
If $\left(c_{n}\right)$ can be chosen to be bounded, then $\varphi$
is called \emph{symmetric} with respect to $\Psi$. Moreover, $\varphi$
is \emph{symmetric on average }with respect to $\Psi$ (\cite[Definition 1.4]{Jaerisch12b})
if 
\[
\sup_{g\in G}\limsup_{n\rightarrow\infty}\frac{\sum_{k=1}^{n}\e^{-kp\mathcal{P}\left(\varphi,\Psi^{-1}\left(\id\right)\cap\Sigma^{*}\right)}\sum_{\omega\in\Sigma^{kp}\cap\Psi^{-1}(g)}\e^{S_{\omega}\varphi}}{\sum_{k=1}^{n}\e^{-kp\mathcal{P}\left(\varphi,\Psi^{-1}\left(\id\right)\cap\Sigma^{*}\right)}\sum_{\omega\in\Sigma^{kp}\cap\Psi^{-1}(g^{-1})}\e^{S_{\omega}\varphi}}<\infty,
\]
where $p:=\gcd\left\{ n\in\N:\exists\omega\in\Sigma^{n}\cap\Psi^{-1}\left(\id\right)\mbox{ such that }\omega_{n}\omega_{1}\in\Sigma^{2}\right\} $. \end{defn}
\begin{rem}
\label{induced-is-gurevich}Throughout this section, we will make
use of the induced pressure $\mathcal{P}\left(\varphi,\Psi^{-1}\left(\id\right)\cap\Sigma^{*}\right)$
(see Definition \ref{def:induced-topological-pressure}), where $\Sigma$
is finitely primitive, $\left(\Sigma\times G,\sigma\rtimes\Psi\right)$
is an irreducible group-extended Markov system and $\varphi:\Sigma\rightarrow\R$
is H\"older continuous. A straightforward generalisation of the proof
of \cite[Remark 5.1.6]{JaerischDissertation11} shows that \foreignlanguage{english}{$\mathcal{P}\left(\varphi,\Psi^{-1}\left(\id\right)\cap\Sigma^{*}\right)$
coincides with the \emph{Gurevi\v c pressure} of $\varphi\circ\pi_{1}$
with respect to }\emph{$\left(\Sigma\times G,\sigma\rtimes\Psi\right)$
}(see \cite{MR1738951}). 
\end{rem}

\subsection{Amenability\label{sub:Amenability-of-the}}

Let us first recall the definition of the important property of groups
which was introduced by von Neumann \cite{vonNeumann1929amenabledef}
under the German name \emph{messbar}. By Day \cite{day1949amenabledef},
groups with this property were renamed amenable groups.
\begin{defn}
A discrete group\emph{ $G$ }is\emph{ amenable} if there exists a
finitely additive probability measure $\nu$ on the power set of $G$,
such that $\nu\left(A\right)=\nu\left(g\left(A\right)\right)$, for
all $g\in G$ and $A\subset G$. 
\end{defn}
The following result is taken from \cite[Corollary 1.6]{Jaerisch12c}.
(See also \cite[Theorem 5.3.11]{JaerischDissertation11} and \cite[Corollary 4.22 and Remark 4.23]{Jaerisch11a}),
where the case of a finite alphabet was considered.) Stadlbauer proved
a similar result for weakly symmetric potentials (\cite[Theorem 4.1]{Stadlbauer11}). 
\begin{thm}
[\cite{Jaerisch12c}, Corollary 1.6] \label{thm:amenable-fullpressure-j}Let
$\Sigma$ be finitely primitive and let $\left(\Sigma\times G,\sigma\rtimes\Psi\right)$
be an irreducible\emph{ }group-extended Markov system. Suppose that
$\varphi:\Sigma\rightarrow\R$ is H\"older continuous with $\mathcal{P}\left(\varphi\right)<\infty$
and that $\varphi$ is asymptotically symmetric with respect to $\Psi$.
If $G$ is amenable then $\mathcal{P}\left(\varphi,\Psi^{-1}\left(\id\right)\cap\Sigma^{*}\right)=\mathcal{P}\left(\varphi\right)$. 
\end{thm}
The next theorem provides a converse of the previous theorem and is
due to Stadlbauer.
\begin{thm}
[\cite{Stadlbauer11}, Theorem 5.4] \label{thm:fullpressure-implies-amenable}Let
$\Sigma$ be finitely primitive and let $\left(\Sigma\times G,\sigma\rtimes\Psi\right)$
be an \linebreak  irreducible\emph{ }group-extended Markov system.
Let $\varphi:\Sigma\rightarrow\R$ be H\"older continuous. \linebreak 
If $\mathcal{P}\left(\varphi,\Psi^{-1}\left(\id\right)\cap\Sigma^{*}\right)=\mathcal{P}\left(\varphi\right)<\infty$,
then $G$ is amenable.
\end{thm}

\subsection{Recurrence and lower bounds for pressure\label{sub:Recurrence-and-lower}}

Let $\Sigma$ be finitely primitive and let $\varphi:\Sigma\rightarrow\R$
be H\"older continuous. Let $\left(\Sigma\times G,\sigma\rtimes\Psi\right)$
denote an irreducible\emph{ }group-extended Markov system. The potential
$\varphi\circ\pi_{1}:\Sigma\times G\rightarrow\R$ is \emph{recurrent}
if $\mathcal{P}\left(\varphi,\Psi^{-1}\left(\id\right)\cap\Sigma^{*}\right)<\infty$
and 
\[
\sum_{n\in\N}\e^{-n\mathcal{P}\left(\varphi,\Psi^{-1}\left(\id\right)\cap\Sigma^{*}\right)}\sum_{\omega\in\Sigma^{n}\cap\Psi^{-1}\left(\id\right)}\e^{S_{\omega}\varphi}=\infty.
\]

\begin{rem*}
It follows from Remark \ref{induced-is-gurevich} that this definition
of a recurrent potential \foreignlanguage{english}{coincides with
Sarig's definition of a recurrent potential (}\cite[Definition 1]{MR1818392}).
\end{rem*}
\selectlanguage{english}%
In order to give lower bounds on $\mathcal{P}\left(\varphi,\Psi^{-1}\left(\id\right)\cap\Sigma^{*}\right)$,
we need the following theorem. 
\selectlanguage{british}%
\begin{thm}
[\cite{Jaerisch12b}, Corollary 1.2, Remark 1.6] \label{thm:recurrence-implies-amenable}Let
$\Sigma$ be finitely primitive and let $\left(\Sigma\times G,\sigma\rtimes\Psi\right)$
be an irreducible\emph{ }group-extended Markov system.\foreignlanguage{english}{\textup{
Let}} $\varphi:\Sigma\rightarrow\R$ be H\"older continuous with
\linebreak \foreignlanguage{english}{$\mathcal{P}\left(\varphi,\Psi^{-1}\left(\id\right)\cap\Sigma^{*}\right)<\infty$.
}If $\varphi\circ\pi_{1}$ is recurrent, then $G$ is amenable. 
\end{thm}
For a recurrent potential, we can characterise when $\mathcal{P}\left(\varphi,\Psi^{-1}\left(\id\right)\cap\Sigma^{*}\right)$
and $\mathcal{P}\left(\varphi\right)$ coincide.
\begin{prop}
[\cite{Jaerisch12b}, Proposition 1.5, Remark 1.6] \label{prop:recurrent-fullpressure-symmetric}Let
$\Sigma$ be finitely primitive and let \linebreak $\left(\Sigma\times G,\sigma\rtimes\Psi\right)$
be an irreducible\emph{ }group-extended Markov system.\foreignlanguage{english}{\textup{
Let}} $\varphi:\Sigma\rightarrow\R$ be H\"older continuous with
\foreignlanguage{english}{$\mathcal{P}\left(\varphi,\Psi^{-1}\left(\id\right)\cap\Sigma^{*}\right)<\infty$.
}If $\varphi\circ\pi_{1}$ is recurrent, then we have that $\mathcal{P}\left(\varphi,\Psi^{-1}\left(\id\right)\cap\Sigma^{*}\right)=\mathcal{P}\left(\varphi\right)$
if and only if $\varphi$ is symmetric on average with respect to
$\Psi$. \end{prop}
\begin{rem*}
By combining Theorem \ref{thm:amenable-fullpressure-j} with Proposition
\ref{prop:recurrent-fullpressure-symmetric}, we see that, if $\varphi\circ\pi_{1}$
is recurrent and if $\varphi$ is asymptotically symmetric with respect
to $\Psi$, then $\varphi$ is symmetric on average with respect to
$\Psi$.
\end{rem*}
The next result gives a lower bound on $\mathcal{P}\left(\varphi,\Psi^{-1}\left(\id\right)\cap\Sigma^{*}\right)$.
A similar result to the first assertion is given in the author's thesis
(\cite[Theorem 5.3.11]{JaerischDissertation11}). The second assertion
is  inspired by \cite[Lemma 5.1]{Jaerisch11a}, where a locally constant
potential $\varphi$ is considered, and makes use of Theorem \ref{thm:recurrence-implies-amenable}. 
\begin{prop}
\label{prop:weak-lower-deltahalf-bound}Let $\Sigma$ be finitely
primitive and let $\left(\Sigma\times G,\sigma\rtimes\Psi\right)$
be an irreducible  group-extended Markov system.\emph{ }For each H\"older
continuous potential $\varphi:\Sigma\rightarrow\R$ the following
holds. 
\begin{enumerate}
\item \label{enu:weak-delta-half-bound}If $\varphi$ is asymptotically
symmetric with respect to $\Psi$ then $2\mathcal{P}\left(\varphi,\Psi^{-1}\left(\id\right)\cap\Sigma^{*}\right)\ge\mathcal{P}\left(2\varphi\right)$. 
\item \label{enu:strong-delta-half-bound}If $\varphi$ is symmetric with
respect to $\Psi$ and $\mathcal{P}\left(2\varphi\right)<\infty$,
then we have that 
\[
2\mathcal{P}\left(\varphi,\Psi^{-1}\left(\id\right)\cap\Sigma^{*}\right)=\mathcal{P}\left(2\varphi\right)\mbox{ if and only if }2\mathcal{P}\left(\varphi\right)=\mathcal{P}\left(2\varphi\right).
\]

\end{enumerate}
\end{prop}
\begin{proof}
We first prove (\ref{enu:weak-delta-half-bound}). Since $\Sigma$
is finitely primitive and $\left(\Sigma\times G,\sigma\rtimes\Psi\right)$
is irreducible, there exists a finite set $B\subset\Psi^{-1}\left(\id\right)\cap\Sigma^{*}$
such that, for all $\omega_{1},\omega_{2}\in\Sigma^{*}$ there is
$\gamma\left(\omega_{1},\omega_{2}\right)\in B$ with $\omega_{1}\gamma\left(\omega_{1},\omega_{2}\right)\omega_{2}\in\Sigma^{*}$.
Define the map  $\Gamma:\Sigma^{*}\times\Sigma^{*}\rightarrow\Sigma^{*}$
given by $\Gamma\left(\omega_{1},\omega_{2}\right):=\omega_{1}\gamma\left(\omega_{1},\omega_{2}\right)\omega_{2}$,
where $\gamma\left(\omega_{1},\omega_{2}\right)\in B$. Note that
the restriction of $\Gamma$ to $\Sigma^{n}\times\Sigma^{*}$ (resp.
$\Sigma^{*}\times\Sigma^{n}$) is at most $\card\left(B\right)$-to-one,
for each $n\in\N$. Setting $C_{B}:=\min\left\{ S_{\gamma}\varphi:\gamma\in B\right\} >-\infty$
and using the bounded distortion property of $\varphi$ with constant
$C_{\varphi}>0$ (see Fact \ref{fact-bounded-distortion-property}),
we have that $S_{\omega_{1}}\varphi+S_{\omega_{2}}\varphi-3C_{\varphi}+C_{B}\le S_{\Gamma\left(\omega_{1},\omega_{2}\right)}\varphi$,
for all $\omega_{1},\omega_{2}\in\Sigma^{*}$. Consequently, setting
$l:=\max\left\{ \left|\gamma\right|:\gamma\in B\right\} $, we obtain
for every sequence $\left(N_{n}\right)\in\N^{\N}$ and $n\in\N$,
\begin{align*}
 & \card\left(B\right)\e^{3C_{\varphi}-C_{B}}\sum_{\omega\in\Sigma^{*}\cap\Psi^{-1}(\id):2n-N_{n}\le\left|\omega\right|\le2n+N_{n}+l}\e^{S_{\omega}\varphi}\\
\ge & \sum_{g\in G}\Bigl(\sum_{\omega_{1}\in\Sigma^{n}\cap\Psi^{-1}(g)}\e^{S_{\omega_{1}}\varphi}\Bigr)\Bigl(\sum_{\omega_{2}\in\Sigma^{*}\cap\Psi^{-1}(g^{-1}):n-N_{n}\le\left|\omega_{2}\right|\le n+N_{n}}\e^{S_{\omega_{2}}\varphi}\Bigr).
\end{align*}
Using that $\varphi$ is asymptotically symmetric with respect to
$\Psi$ with $n_{0}\in\N$ and sequences $\left(c_{n}\right)\in\R^{\N}$
and $\left(N_{n}\right)\in\N^{\N}$ as in Definition \ref{def-asympt-symmetric},
it follows from the previous inequality that for all $n\ge n_{0}$,
\begin{align}
 & \quad\card\left(B\right)\e^{3C_{\varphi}-C_{B}}c_{n}\sum_{\omega\in\Sigma^{*}\cap\Psi^{-1}(\id):2n-N_{n}\le\left|\omega\right|\le2n+N_{n}+l}\e^{S_{\omega}\varphi}\nonumber \\
\ge & \sum_{g\in G}\Bigl(\sum_{\omega_{1}\in\Sigma^{n}\cap\Psi^{-1}(g)}\e^{S_{\omega_{1}}\varphi}\Bigr)\Bigl(\sum_{\omega_{2}\in\Sigma^{n}\cap\Psi^{-1}(g)}\e^{S_{\omega_{2}}\varphi}\Bigr)\ge\sum_{g\in G}\Bigl(\sum_{\omega_{1}\in\Sigma^{n}\cap\Psi^{-1}(g)}\e^{2S_{\omega_{1}}\varphi}\Bigr)=\sum_{\omega\in\Sigma^{n}}\e^{2S_{\omega}\varphi}.\label{eq:poincare-series}
\end{align}
Using that $\lim_{n}(c_{n})^{1/n}=1$ it follows from (\ref{eq:poincare-series})
that 
\[
\limsup_{n\rightarrow\infty}\frac{1}{n}\log\sum_{\omega\in\Sigma^{*}\cap\Psi^{-1}(\id):2n-N_{n}\le\left|\omega\right|\le2n+N_{n}+l}\e^{S_{\omega}\varphi}\ge\limsup_{n\rightarrow\infty}\frac{1}{n}\log\sum_{\omega\in\Sigma^{n}}\e^{2S_{\omega}\varphi}=\mathcal{P}\left(2\varphi\right).
\]
Finally, using that $\lim_{n}n^{-1}N_{n}=0$, one verifies that 
\[
2\mathcal{P}\left(\varphi,\Psi^{-1}\left\{ \id\right\} \cap\Sigma^{*}\right)\ge\limsup_{n\rightarrow\infty}\frac{1}{n}\log\sum_{\omega\in\Sigma^{*}\cap\Psi^{-1}(\id):2n-N_{n}\le\left|\omega\right|\le2n+N_{n}+l}\e^{S_{\omega}\varphi},
\]
which finishes the proof of (\ref{enu:weak-delta-half-bound}).

We now turn to the proof of (\ref{enu:strong-delta-half-bound}).
First note that, by passing to the potential $\varphi-\mathcal{P}\left(2\varphi\right)/2$,
we may without loss of generality assume that $\mathcal{P}\left(2\varphi\right)=0$.
It remains to show that $\mathcal{P}\left(\varphi,\Psi^{-1}\left(\id\right)\cap\Sigma^{*}\right)=0$
if and only if $\mathcal{P}\left(\varphi\right)=0$. Since $\mathcal{P}\left(\varphi,\Psi^{-1}\left(\id\right)\cap\Sigma^{*}\right)\ge\mathcal{P}\left(2\varphi\right)/2=0$
by (\ref{enu:weak-delta-half-bound}), we deduce that $\mathcal{P}\left(\varphi\right)=0$
implies $\mathcal{P}\left(\varphi,\Psi^{-1}\left(\id\right)\cap\Sigma^{*}\right)=0$.
Now, for the opposite implication, suppose that $\mathcal{P}\left(\varphi,\Psi^{-1}\left(\id\right)\cap\Sigma^{*}\right)=0$.
Since $\Sigma$ is finitely primitive and $\mathcal{P}\left(2\varphi\right)=0$
there exists a unique $\sigma$-invariant Gibbs measure $\mu$ for
$2\varphi$ by Theorem \ref{thm:existence-of-gibbs-measures}. By
(\ref{eq:gibbs-equation}) there exists a constant $C_{\mu}>0$ such
that for all $n\in\N$, we have 
\[
\sum_{\omega\in\Sigma^{n}}\e^{2S_{\omega}\varphi}\ge C_{\mu}^{-1}\sum_{\omega\in\Sigma^{n}}\mu\left(\left[\omega\right]\right)=C_{\mu}^{-1}>0.
\]
Since $\varphi$ is symmetric with respect to $\Psi$ and by (\ref{eq:poincare-series}),
there exists $n_{0}\in\N$ and $C>0$, such that for all $n\ge n_{0}\in\N$,
\begin{equation}
\sum_{\omega\in\Sigma^{*}\cap\Psi^{-1}\left(\id\right):2n-N_{n}\le\left|\omega\right|\le2n+N_{n}+l}\e^{S_{\omega}\varphi}\ge C\sum_{\omega\in\Sigma^{n}}\e^{2S_{\omega}\varphi}\ge CC_{\mu}^{-1}>0.\label{eq:delta-half-divergence-type-pre}
\end{equation}
Since $\lim_{n}n^{-1}N_{n}=0$, there exists a sequence $\left(n_{k}\right)\in\N^{\N}$
tending to infinity, such that the sets $\bigl\{2n_{k}-N_{n_{k}},\dots,2n_{k}+N_{n_{k}}+l\bigr\}_{k\in\N}$
are pairwise disjoint. Hence, by (\ref{eq:delta-half-divergence-type-pre}),
we have that 
\[
\sum_{n\in\N}\sum_{\omega\in\Sigma^{n}\cap\Psi^{-1}\left(\id\right)}\e^{S_{\omega}\varphi}=\infty.
\]
Since $\mathcal{P}\left(\varphi,\Psi^{-1}\left(\id\right)\cap\Sigma^{*}\right)=0$,
we have thus shown that $\varphi\circ\pi_{1}$ is a recurrent. Hence,
Theorem \ref{thm:recurrence-implies-amenable} gives that $G$ is
amenable. Finally, it follows from Theorem \ref{thm:amenable-fullpressure-j}
that $0=\mathcal{P}\left(\varphi,\Psi^{-1}\left(\id\right)\cap\Sigma^{*}\right)=\mathcal{P}\left(\varphi\right)$,
which finishes the proof of (\ref{enu:strong-delta-half-bound}). \end{proof}
\begin{cor}
\label{cor:non-amenable-stricthalfbound}Let $\Sigma$ be finitely
primitive and let $\left(\Sigma\times G,\sigma\rtimes\Psi\right)$
be an irreducible  group-extended Markov system.\emph{ }Let $\varphi:\Sigma\rightarrow\R$
be H\"older continuous and symmetric with respect to $\Psi$. If
$G$ is non-amenable, then $2\mathcal{P}\left(\varphi,\Psi^{-1}\left(\id\right)\cap\Sigma^{*}\right)>\mathcal{P}\left(2\varphi\right)$. \end{cor}
\begin{proof}
Suppose for a contradiction that the claim is false. Then, by Proposition
\ref{prop:weak-lower-deltahalf-bound} (\ref{enu:weak-delta-half-bound})
and (\ref{enu:strong-delta-half-bound}), we have $2\mathcal{P}\left(\varphi,\Psi^{-1}\left(\id\right)\cap\Sigma^{*}\right)=\mathcal{P}\left(2\varphi\right)=2\mathcal{P}\left(\varphi\right)$.
By Theorem \ref{thm:fullpressure-implies-amenable} we conclude that
$G$ is amenable, which is a contradiction.
\end{proof}

\section{Proof of the main results\label{sec:Proof}}

For a conformal GDMS $\Phi$ associated to $\F_{d}=\langle g_{1},\dots,g_{d}\rangle$
with $d\ge2$, set $I:=\left\{ g_{1},g_{1}^{-1},\dots,g_{d},g_{d}^{-1}\right\} $
and $\Sigma:=\bigl\{\tau\in I^{\N}:\forall i\in\N\,\,\tau_{i}\neq\tau_{i+1}^{-1}\bigr\}$.
One immediately verifies that the Markov shift $\Sigma$ is finitely
primitive. For a non-trivial normal subgroup $N$ of $\F_{d}$, let
$\Psi_{N}:I^{*}\rightarrow\F_{d}/N$ denote the canonical semigroup
homomorphism given by $\Psi_{N}\left(g\right)=g\mbox{ mod }N$, for
each $g\in I$. Using that $d\ge2$ and that $N$ is a non-trivial
normal subgroup of $\F_{d}$, we see that the group-extended Markov
system $\sigma\rtimes\Psi_{N}:\Sigma\times(\F_{d}/N)\rightarrow\Sigma\times(\F_{d}/N)$
is irreducible. We consider $\NwithoutId$ as a subset of $\Sigma^{*}$.
To apply the results of Section \ref{sec:Group-Extended-Markov},
we will frequently make use of the fact that 
\[
\NwithoutId=\Psi_{N}^{-1}\left(\id\right)\cap\Sigma^{*}.
\]

The geometric potential $\zeta$ of $\Phi$ is H\"older continuous
by Fact \ref{fact-geometricpotential-is-hoelder}. Since the Lipschitz
constants of $\Phi$ are bounded away from one, we have that $\sup_{\tau\in\Sigma}\zeta\left(\tau\right)<0$.
Since $\card\left(I\right)<\infty$ we have the following by Fact
\ref{fac:criticalexponents-via-pressure}. 
\begin{fact}
\label{t-characterisation} Let $\psi:\Sigma\rightarrow\R$ be H\"older
continuous and let $t_{N}:\R\rightarrow\R$ denote the free energy
function of $\left(N,\Phi,\psi\right)$. For each $\beta\in\R$, the
function $u\mapsto\mathcal{P}\left(\beta\psi+u\zeta,\NwithoutId\right)$
is strictly decreasing with values in $\R$, and we have $\mathcal{P}\left(\beta\psi+t_{N}\left(\beta\right)\zeta,\NwithoutId\right)=0$. \end{fact}
\begin{lem}
\label{lem:tzero-is-positive}We have $\delta_{N}>0$. \end{lem}
\begin{proof}
Since $0:\Sigma\rightarrow\left\{ 0\right\} $ is asymptotically symmetric
with respect to $\Psi_{N}$, we have 
\[
2\mathcal{P}\left(0,\NwithoutId\right)\ge\mathcal{P}\left(0\right)=\log\left(2d-1\right)>0
\]
by Proposition \ref{prop:weak-lower-deltahalf-bound} (\ref{enu:weak-delta-half-bound}).
Hence, we have $\delta_{N}=t_{N}\left(0\right)>0$ by Fact \ref{t-characterisation}. 
\end{proof}
We will repeatedly make use of the following fact about the convex
conjugate of a convex function. For the proof we refer to \cite[Theorem 23.5 and Corollary 26.4.1]{rockafellar-convexanalysisMR0274683}.
\begin{fact}
\label{convexconjugate-facts}Let $f:\R\rightarrow\R$ be a convex
function and let $f^{*}:\R\rightarrow\R\cup\left\{ \infty\right\} $
denote the convex conjugate of $f$. 
\begin{enumerate}
\item \label{enu:conjugate-facts-1}Let $\beta\in\R$. If $\alpha\in\partial f\left(\beta\right)$
then $f^{*}\left(\alpha\right)=\alpha\beta-f\left(\beta\right)$. 
\item \label{enu:conjugate-facts-2}$\Int\left\{ x'\in\R:f^{*}\left(x'\right)<\infty\right\} \subset\left\{ x\in\R:\partial f\left(x\right)\neq\emptyset\right\} \subset\left\{ x'\in\R:f^{*}\left(x'\right)<\infty\right\} $.
In particular, if $\alpha\notin\overline{\partial f\left(\R\right)}$
then $f^{*}\left(\alpha\right)=\infty$. 
\end{enumerate}
\end{fact}
The following results about the free energy function are crucial to
derive Theorem \ref{thm:maintheorem-limitset} and \ref{thm:maintheorem}.
\begin{prop}
\label{prop:free-energy-conjugate-mainresults}Let $\Phi$ denote
a conformal GDMS associated to \emph{$\F_{d}$ with} $d\ge2$, and
let $N$ denote a non-trivial normal subgroup of $\F_{d}$. Let $\psi:\Sigma\rightarrow\R$
be H\"older continuous and let $t_{N}:\R\rightarrow\R$ and $t:\R\rightarrow\R$
denote the free energy functions of $\left(N,\Phi,\psi\right)$ and
$\left(\F_{d},\Phi,\psi\right)$ respectively. Then we have the following.
\begin{enumerate}
\item \label{enu:t-mainprop-1}

\begin{enumerate}
\item \label{enu:t-mainprop-1a}We have $t_{N}\left(\beta\right)\le t\left(\beta\right)$
and $-t_{N}^{*}\left(-\alpha\right)\le-t^{*}\left(-\alpha\right)$
for all $\beta,\alpha\in\R$. 
\item \label{enu:t-mainprop-1b}If $t_{N}\left(\beta\right)=t\left(\beta\right)$
for some $\beta\in\R$, then $\F_{d}/N$ is amenable. 
\item \label{enu:t-mainprop-1c}If $-t_{N}^{*}\left(-\alpha\right)=-t^{*}\left(-\alpha\right)$
for some $\alpha\in-\partial t\left(\R\right)$, then $\F_{d}/N$
is amenable.
\item \label{enu:t-mainprop-1d}$\Int\left(\partial t_{N}\left(\R\right)\right)\subset\Int\left(\partial t\left(\R\right)\right)$. 
\end{enumerate}
\item Suppose that $\left(N,\Phi,\psi\right)$ is asymptotically symmetric.
\label{enu:t-mainprop-2}

\begin{enumerate}
\item \label{enu:t-mainprop-2a}If $\F_{d}/N$ is amenable, then $t_{N}\left(\beta\right)=t\left(\beta\right)$
and $-t_{N}^{*}\left(-\alpha\right)=-t^{*}\left(-\alpha\right)$ for
all $\beta,\alpha\in\R$.
\item \label{enu:t-mainprop-2b}We have $2t_{N}\left(\beta\right)\ge t\left(2\beta\right)$
and $-t_{N}^{*}\left(-\alpha\right)\ge-t^{*}\left(-\alpha\right)/2$
for all $\beta,\alpha\in\R$.
\item \label{enu:t-mainprop-2c}If $\left(N,\Phi,\psi\right)$ is symmetric
and $\F_{d}/N$ is non-amenable, then we have $2t_{N}\left(\beta\right)>t\left(2\beta\right)$
and $-t_{N}^{*}\left(-\alpha\right)>-t^{*}\left(-\alpha\right)/2$,
for every $\beta\in\R$ and $\alpha\in-\partial t_{N}\left(\R\right)$. 
\item \label{enu:t-mainprop-2d}$\Int\left(\partial t_{N}\left(\R\right)\right)=\Int\left(\partial t\left(\R\right)\right)$. 
\end{enumerate}
\item Let $\beta\in\R$ and suppose that $\left(N,\Phi,\psi\right)$ is
of divergence type in $\beta$.\label{enu:t-mainprop-3}

\begin{enumerate}
\item \label{enu:t-mainprop-3a}Then $\F_{d}/N$ is amenable.
\item \label{enu:t-mainprop-3b}$t_{N}\left(\beta\right)=t\left(\beta\right)$
if and only if $\left(N,\Phi,\psi\right)$ is symmetric on average
in $\beta$. 
\item \label{enu:t-mainprop-3c}If $\alpha\in-\partial t(\beta)$ and $-t_{N}^{*}\left(-\alpha\right)=-t^{*}\left(-\alpha\right)$,
then $\left(N,\Phi,\psi\right)$ is symmetric on average in $\beta$. 
\item \label{enu:t-mainprop-3d}If $\left(N,\Phi,\psi\right)$ is symmetric
on average in $\beta$, then $-t_{N}^{*}\left(-\alpha\right)=-t^{*}\left(-\alpha\right)$
for $\alpha\in-\partial t_{N}(\beta)$. 
\end{enumerate}
\end{enumerate}
\end{prop}
\begin{proof}
The first assertion in (\ref{enu:t-mainprop-1a}) follows from Fact
\ref{t-characterisation}, since we have 
\[
\mathcal{P}\left(\beta\psi+t\left(\beta\right)\zeta,\NwithoutId\right)\le\mathcal{P}\left(\beta\psi+t\left(\beta\right)\zeta\right)=0,\mbox{ for }\beta\in\R.
\]
By the definition of the convex conjugate, we then have for $\alpha\in\R$,
\[
-t_{N}^{*}\left(-\alpha\right)=\inf_{\beta\in\R}\left\{ t_{N}\left(\beta\right)+\beta\alpha\right\} \le\inf_{\beta\in\R}\left\{ t\left(\beta\right)+\beta\alpha\right\} =-t^{*}\left(-\alpha\right).
\]
To prove (\ref{enu:t-mainprop-1b}), suppose that $t_{N}\left(\beta\right)=t\left(\beta\right)$
for some $\beta\in\R$. Then by Fact \ref{t-characterisation} we
have that \linebreak  $\mathcal{P}\left(\beta\psi+t_{N}\left(\beta\right)\zeta,\NwithoutId\right)=\mathcal{P}\left(\beta\psi+t\left(\beta\right)\zeta\right)=0$.
Applying Theorem \ref{thm:fullpressure-implies-amenable} to the H\"older
continuous  potential $\beta\psi+t_{N}\left(\beta\right)\zeta:\Sigma\rightarrow\R$
and the group-extended Markov system $\left(\Sigma\times(F_{d}/N),\sigma\rtimes\Psi_{N}\right)$
gives  that $\F_{d}/N$ is amenable. For the proof of (\ref{enu:t-mainprop-1c}),
let $\alpha\in-\partial t\left(\beta\right)$ for some $\beta\in\R$.
By Fact \ref{convexconjugate-facts} (\ref{enu:conjugate-facts-1}),
the first inequality in (\ref{enu:t-mainprop-1a}) and the definition
of the convex conjugate, we have 
\[
-t^{*}\left(-\alpha\right)=t(\beta)+\beta\alpha\ge t_{N}\left(\beta\right)+\beta\alpha\ge-t_{N}^{*}\left(-\alpha\right).
\]
Consequently, if $-t_{N}^{*}\left(-\alpha\right)=-t^{*}\left(-\alpha\right)$,
then we have $t_{N}\left(\beta\right)=t(\beta)$ and amenability of
$\F_{d}/N$ follows from (\ref{enu:t-mainprop-1b}). To prove (\ref{enu:t-mainprop-1d}),
let $\alpha\in\Int\left(\partial t_{N}\left(\R\right)\right)$. By
Fact \ref{convexconjugate-facts} (\ref{enu:conjugate-facts-2}) we
have $t_{N}^{*}\left(\alpha\right)<\infty$. By (\ref{enu:t-mainprop-1a})
we have $t^{*}\left(\alpha\right)\le t_{N}^{*}\left(\alpha\right)<\infty$.
By Fact \ref{convexconjugate-facts} (\ref{enu:conjugate-facts-2})
again, we conclude that $\alpha\in\overline{\partial t\left(\R\right)}$.
Since $\alpha$ is an interior point of $\partial t_{N}\left(\R\right)$,
we have thus shown that $\Int\left(\partial t_{N}\left(\R\right)\right)\subset\Int\left(\partial t\left(\R\right)\right)$. 

Now suppose that $\left(N,\Phi,\psi\right)$ is asymptotically symmetric,
that is, $\beta\psi+u\zeta$ is asymptotically symmetric with respect
to $\Psi_{N}$, for all $\beta,u\in\R$. To prove (\ref{enu:t-mainprop-2a}),
suppose that $\F_{d}/N$ is amenable. By applying Theorem \ref{thm:amenable-fullpressure-j}
to the H\"older continuous  potential $\beta\psi+t\left(\beta\right)\zeta$
and the group-extended Markov system $\left(\Sigma\times(F_{d}/N),\sigma\rtimes\Psi_{N}\right)$,
we obtain by Fact \ref{t-characterisation} that $\mathcal{P}\left(\beta\psi+t\left(\beta\right)\zeta,\NwithoutId\right)=\mathcal{P}\left(\beta\psi+t\left(\beta\right)\zeta\right)=0$
for $\beta\in\R$. Consequently, we have $t_{N}\left(\beta\right)=t\left(\beta\right)$
for all $\beta\in\R$ by Fact \ref{t-characterisation} and thus,
$-t_{N}^{*}\left(-\alpha\right)=-t^{*}\left(-\alpha\right)$ for all
$\alpha\in\R$. To prove (\ref{enu:t-mainprop-2b}) let $\beta\in\R$.
\foreignlanguage{english}{Applying Proposition \ref{prop:weak-lower-deltahalf-bound}
(}\ref{enu:weak-delta-half-bound}\foreignlanguage{english}{) to the
asymptotically symmetric }H\"older continuous  potential\foreignlanguage{english}{
$\beta\psi+\left(t\left(2\beta\right)/2\right)\zeta$ gives that 
\begin{equation}
2\mathcal{P}\left(\beta\psi+\left(t\left(2\beta\right)/2\right)\zeta,\NwithoutId\right)\ge\mathcal{P}\left(2\beta\psi+t\left(2\beta\right)\zeta\right)=0.\label{eq:mainproof-lowerbound}
\end{equation}
Hence, we have $t_{N}\left(\beta\right)\ge t\left(2\beta\right)/2$
}by Fact \ref{t-characterisation}\foreignlanguage{english}{ and we
obtain for $\alpha\in\R$, 
\[
-t_{N}^{*}\left(-\alpha\right)=\inf_{\beta\in\R}\left\{ t_{N}\left(\beta\right)+\beta\alpha\right\} \ge\frac{1}{2}\inf_{\beta\in\R}\left\{ t\left(2\beta\right)+2\beta\alpha\right\} =-t^{*}\left(-\alpha\right)/2,
\]
which finishes the proof of (}\ref{enu:t-mainprop-2b}\foreignlanguage{english}{).
To prove (}\ref{enu:t-mainprop-2c}\foreignlanguage{english}{) suppose
that $\left(N,\Phi,\psi\right)$ is symmetric and that} $\F_{d}/N$
is non-amenable. Then the inequality in (\ref{eq:mainproof-lowerbound})
is strict by Corollary \ref{cor:non-amenable-stricthalfbound}. Hence,
\foreignlanguage{english}{$t_{N}\left(\beta\right)>t\left(2\beta\right)/2$
for every $\beta\in\R$. Moreover, if $\alpha\in-\partial t_{N}\left(\beta\right)$
for some $\beta\in\R$, then }by using Fact \ref{convexconjugate-facts}
(\ref{enu:conjugate-facts-1}) we deduce that \foreignlanguage{english}{
\[
-t_{N}^{*}\left(-\alpha\right)=t_{N}\left(\beta\right)+\beta\alpha>\left(t\left(2\beta\right)+2\beta\alpha\right)\big/2\ge-t^{*}\left(-\alpha\right)/2.
\]
To prove (}\ref{enu:t-mainprop-2d}\foreignlanguage{english}{), it
suffices to show that $\partial t\left(\R\right)\subset\overline{\partial t_{N}\left(\R\right)}$.
Then $\Int\left(\partial t_{N}\left(\R\right)\right)=\Int\left(\partial t\left(\R\right)\right)$
follows by combining with (\ref{enu:t-mainprop-1d}). Let $\alpha\in-\partial t\left(\beta\right)$
for some $\beta\in\R$. Then by (}\ref{enu:t-mainprop-2b}\foreignlanguage{english}{)
and} Fact \ref{convexconjugate-facts} (\ref{enu:conjugate-facts-1})\foreignlanguage{english}{
we have 
\[
-t_{N}^{*}\left(-\alpha\right)\ge-t^{*}\left(-\alpha\right)\big/2=\left(t(\beta)+\beta\alpha\right)\big/2>-\infty,
\]
which shows that $\alpha\in-\overline{\partial t_{N}\left(\R\right)}$
by }Fact \ref{convexconjugate-facts} (\ref{enu:conjugate-facts-2}).

To prove (\ref{enu:t-mainprop-3}) let $\beta\in\R$ and suppose that
$\left(N,\Phi,\psi\right)$ is of divergence type\foreignlanguage{english}{
in $\beta$, that is, }$\left(\beta\psi+t_{N}\left(\beta\right)\zeta\right)\circ\pi_{1}$
is recurrent with respect to $\sigma\rtimes\Psi_{N}$\foreignlanguage{english}{.
Then amenability of $\F_{d}/N$ follows from Theorem \ref{thm:recurrence-implies-amenable},
which proves (}\ref{enu:t-mainprop-3a}\foreignlanguage{english}{).
To prove (}\ref{enu:t-mainprop-3b}\foreignlanguage{english}{), first
observe that $\left(N,\Phi,\psi\right)$ is symmetric on average in
$\beta$ if and only if $\beta\psi+t\left(\beta\right)\zeta$ is symmetric
on average with respect to $\Psi_{N}$. Now, the equivalence in (}\ref{enu:t-mainprop-3b}\foreignlanguage{english}{)
follows from Proposition \ref{prop:recurrent-fullpressure-symmetric}
and Fact \ref{t-characterisation}. To prove (}\ref{enu:t-mainprop-3c}\foreignlanguage{english}{),
suppose that $-t_{N}^{*}\left(-\alpha\right)=-t^{*}\left(-\alpha\right)$
for $\alpha\in-\partial t(\beta)$. As in the proof of (}\ref{enu:t-mainprop-1c}\foreignlanguage{english}{)
we deduce that $t_{N}\left(\beta\right)=t(\beta)$, which then implies
that} $\left(N,\Phi,\psi\right)$ is symmetric on average in $\beta$
by (\ref{enu:t-mainprop-3b}). In order to prove (\ref{enu:t-mainprop-3d}),
suppose that $\left(N,\Phi,\psi\right)$ is symmetric on average in
$\beta$ and let $\alpha\in-\partial t_{N}(\beta)$. By (\ref{enu:t-mainprop-3b})
we then have $t_{N}\left(\beta\right)=t(\beta)$. Hence, we obtain
by Fact \ref{convexconjugate-facts} (\ref{enu:conjugate-facts-1})\foreignlanguage{english}{
}that 
\[
-t_{N}^{*}\left(-\alpha\right)=t_{N}(\beta)+\beta\alpha=t\left(\beta\right)+\beta\alpha\ge-t^{*}\left(-\alpha\right).
\]
Combining with (\ref{enu:t-mainprop-1a}) finishes the proof of (\ref{enu:t-mainprop-3d})
and completes the proof of the proposition. 
\end{proof}
\selectlanguage{english}%
We are now in the position to prove the main theorems. 
\selectlanguage{british}%
\begin{proof}
[Proof of Theorem $\ref{thm:maintheorem-limitset}$ ]The first assertion
follows from Proposition \ref{prop:t-vs-dimension} (\ref{enu:t0-dimension-limitset})
and Lemma \ref{lem:tzero-is-positive}. To apply Proposition \ref{prop:free-energy-conjugate-mainresults}
in what follows, recall that $\delta_{N}=t_{N}\left(0\right)$ and
$\delta=t\left(0\right)$, where $t_{N}:\R\rightarrow\R$ and $t:\R\rightarrow\R$
denote the free energy functions of $\left(N,\Phi,0\right)$ and $\left(\F_{d},\Phi,0\right)$
respectively. To prove (\ref{enu:maintheorem-limitset-nonamenable-implies-gap-1}),
suppose that $\F_{d}/N$ is non-amenable. By (\ref{enu:t-mainprop-1a})
and (\ref{enu:t-mainprop-1b}) of Proposition \ref{prop:free-energy-conjugate-mainresults},
we have $\delta_{N}<\delta$. Combining with (\ref{enu:maintheorem-0}),
finishes the proof of (\ref{enu:maintheorem-limitset-nonamenable-implies-gap-1}). 

For the proof of (\ref{enu:maintheorem-limitset-2}), suppose that
$\left(N,\Phi,0\right)$ is asymptotically symmetric. If $\F_{d}/N$
is amenable, then we have $\delta_{N}=\delta$ by Proposition \ref{prop:free-energy-conjugate-mainresults}
(\ref{enu:t-mainprop-2a}). Now suppose that $\F_{d}/N$ is non-amenable.
Using Proposition \ref{prop:free-energy-conjugate-mainresults} (\ref{enu:t-mainprop-2b})
we obtain that $\delta_{N}\ge\delta/2$, and if $\left(N,\Phi,0\right)$
is symmetric, then we have $\delta_{N}>\delta/2$ by Proposition \ref{prop:free-energy-conjugate-mainresults}
(\ref{enu:t-mainprop-2c}). In light of (\ref{enu:maintheorem-0})
the proof of (\ref{enu:maintheorem-limitset-2}) is complete. 

To prove (\ref{enu:maintheorem-limitset-3}), suppose that $\left(N,\Phi\right)$
is of divergence type, that is, $\left(N,\Phi,0\right)$ is of divergence
type in $0$. Then the assertion in (\ref{enu:maintheorem-limitset-divtype-amenable})
follows from Proposition \ref{prop:free-energy-conjugate-mainresults}
(\ref{enu:t-mainprop-3a}) and the assertion in (\ref{enu:maintheorem-limitset-divtype-full-symmaverage})
follows from Proposition \ref{prop:free-energy-conjugate-mainresults}
(\ref{enu:t-mainprop-3b}) for $\beta=0$. The proof is complete.
\end{proof}

\begin{proof}
[Proof of Theorem $\ref{thm:maintheorem}$ ]The first assertion follows
from Proposition \ref{prop:t-vs-dimension} (\ref{enu:multifractal-formalism})
and Proposition \ref{prop:free-energy-conjugate-mainresults} (\ref{enu:t-mainprop-1d})
and (\ref{enu:t-mainprop-2d}). To prove (\ref{enu:maintheorem-nonamenable-implies-gap}),
suppose that $\F_{d}/N$ is non-amenable and let $\alpha\in\left(\alpha_{-},\alpha_{+}\right)$.
Since $\left(\alpha_{-},\alpha_{+}\right)=-\Int\left(\partial t\left(\R\right)\right)$
we have $-t_{N}^{*}\left(-\alpha\right)<-t^{*}\left(-\alpha\right)$
by (\ref{enu:t-mainprop-1a}) and (\ref{enu:t-mainprop-1c}) of Proposition
\ref{prop:free-energy-conjugate-mainresults} . Further, by Proposition
\ref{prop:t-vs-dimension} (\ref{enu:multifractal-formalism}), we
have $\dim_{H}\left(\pi_{\Phi}\left(\mathcal{F}\left(\alpha,\Phi,\psi\right)\right)\right)=-t^{*}\left(-\alpha\right)>0$.
Then, by Proposition \ref{prop:t-vs-dimension} (\ref{enu:multifractal-upperbound}),
we obtain that 
\[
\dim_{H}\left(\pi_{\Phi}\left(\Lambda_{\mathrm{r}}\left(N\right)\cap\mathcal{F}\left(\alpha,\Phi,\psi\right)\right)\right)\le\max\left\{ -t_{N}^{*}\left(-\alpha\right),0\right\} <-t^{*}\left(-\alpha\right)=\dim_{H}\left(\pi_{\Phi}\left(\mathcal{F}\left(\alpha,\Phi,\psi\right)\right)\right),
\]
which gives the desired inequality in (\ref{enu:maintheorem-nonamenable-implies-gap}).

Now suppose that $\left(N,\Phi,\psi\right)$ is asymptotically symmetric
and let $\alpha\in\left(\alpha_{-},\alpha_{+}\right)$. By Proposition
\ref{prop:free-energy-conjugate-mainresults} (\ref{enu:t-mainprop-2d}),
we have  $\Int\left(\partial t_{N}\left(\R\right)\right)=\Int\left(\partial t\left(\R\right)\right)$.
Hence, we have $\dim_{H}\left(\pi_{\Phi}\left(\Lambda_{\mathrm{r}}\left(N\right)\cap\mathcal{F}\left(\alpha,\Phi,\psi\right)\right)\right)=-t_{N}^{*}\left(-\alpha\right)$
and $\dim_{H}\left(\pi_{\Phi}\left(\mathcal{F}\left(\alpha,\Phi,\psi\right)\right)\right)=-t^{*}\left(-\alpha\right)$
by Proposition \ref{prop:t-vs-dimension} (\ref{enu:multifractal-formalism}).
The assertion in (\ref{enu:maintheorem-amenablesymmetric-fulldimension})
is then a consequence of Proposition \ref{prop:free-energy-conjugate-mainresults}
(\ref{enu:t-mainprop-2a}). The lower bound in (\ref{enu:maintheorem-symmetriclowerbound})
is deduced  from Proposition \ref{prop:free-energy-conjugate-mainresults}
(\ref{enu:t-mainprop-2b}). Finally, if $\left(N,\Phi,\psi\right)$
is symmetric, then the strict inequality in (\ref{enu:maintheorem-symmetriclowerbound})
follows from Proposition \ref{prop:free-energy-conjugate-mainresults}
(\ref{enu:t-mainprop-2c}). 

Let us now turn to the proof of (\ref{enu:maintheorem-3}). The assertion
in (\ref{enu:maintheorem-limitset-divtype-amenable}) follows from
Proposition \ref{prop:free-energy-conjugate-mainresults} (\ref{enu:t-mainprop-3a}).
To prove (\ref{enu:maintheorem-divtype-full-implies-symmaverage}),
let $\alpha\in-\partial t(\beta)$ and suppose that $\dim_{H}\left(\pi_{\Phi}\left(\Lambda_{\mathrm{r}}\left(N\right)\cap\mathcal{F}\left(\alpha,\Phi,\psi\right)\right)\right)=\dim_{H}\left(\pi_{\Phi}\left(\mathcal{F}\left(\alpha,\Phi,\psi\right)\right)\right)$.
By Proposition \ref{prop:t-vs-dimension} (\ref{enu:multifractal-upperbound})
and (\ref{enu:multifractal-formalism}) we then have that 
\begin{align*}
0<-t^{*}\left(-\alpha\right) & =\dim_{H}\left(\pi_{\Phi}\left(\mathcal{F}\left(\alpha,\Phi,\psi\right)\right)\right)=\dim_{H}\left(\pi_{\Phi}\left(\Lambda_{\mathrm{r}}\left(N\right)\cap\mathcal{F}\left(\alpha,\Phi,\psi\right)\right)\right)\\
 & \le\max\left\{ -t_{N}^{*}\left(-\alpha\right),0\right\} .
\end{align*}
Hence, we have $-t_{N}^{*}\left(-\alpha\right)=-t^{*}\left(-\alpha\right)$
and Proposition \ref{prop:free-energy-conjugate-mainresults} (\ref{enu:t-mainprop-3c})
gives that $\left(N,\Phi,\psi\right)$ is symmetric on average in
$\beta$. Finally, to prove (\ref{enu:maintheorem-divtype-symmaverage-implies-full}),
let $\alpha\in-\left(\partial t_{N}\left(\beta\right)\cap\Int\left(\partial t_{N}\left(\R\right)\right)\right)$.
Suppose that $\left(N,\Phi,\psi\right)$ is symmetric on average in
$\beta$. Then we have $-t_{N}^{*}\left(-\alpha\right)=-t^{*}\left(-\alpha\right)$
by Proposition \ref{prop:free-energy-conjugate-mainresults} (\ref{enu:t-mainprop-3d}).
Since $-t^{*}\left(-\alpha\right)=\dim_{H}\left(\pi_{\Phi}\left(\mathcal{F}\left(\alpha,\Phi,\psi\right)\right)\right)$
and $-t_{N}^{*}\left(-\alpha\right)=\dim_{H}\left(\pi_{\Phi}\left(\Lambda_{\mathrm{r}}\left(N\right)\cap\mathcal{F}\left(\alpha,\Phi,\psi\right)\right)\right)$
by Proposition \ref{prop:t-vs-dimension} (\ref{enu:multifractal-formalism})
and Proposition \ref{prop:free-energy-conjugate-mainresults} (\ref{enu:t-mainprop-1d}),
the proof is complete. 
\end{proof}
\providecommand{\bysame}{\leavevmode\hbox to3em{\hrulefill}\thinspace}
\providecommand{\MR}{\relax\ifhmode\unskip\space\fi MR }
\providecommand{\MRhref}[2]{%
  \href{http://www.ams.org/mathscinet-getitem?mr=#1}{#2}
}
\providecommand{\href}[2]{#2}

\end{document}